\documentclass[leqno]{amsart}
\usepackage{amssymb}
\usepackage{mathrsfs}
\usepackage{amsmath, amsfonts, vmargin, enumerate}
\usepackage{graphicx}
\usepackage{color}
\usepackage{verbatim}
\usepackage{amsthm}
\usepackage{latexsym, bm}
\usepackage[latin1]{inputenc}
\usepackage[T1]{fontenc}

\usepackage{euscript}
\usepackage{dsfont}

\input xypic

\xyoption {all}

 \makeatletter

\newtheorem{thm}{Theorem}

\newtheorem{cor}[thm]{Corollary}
\newtheorem{lem}[thm]{Lemma}
\newtheorem{defi}[thm]{Definition}
\newtheorem{remark}[thm]{Remark}
\newtheorem{example}[thm]{Example}
\newtheorem{pb}[thm]{Problem}
\newtheorem{conj}[thm]{Conjecture}

\newenvironment{rk}{\begin{remark}\rm}{\end{remark}}

\newenvironment{problem}{\begin{pb}\rm}{\end{pb}}

\newcommand{\real}{{\mathbb R}}
\newcommand{\nat}{{\mathbb N}}
\newcommand{\ent}{{\mathbb Z}}
\newcommand{\com}{{\mathbb C}}
\newcommand{\un}{{\mathds {1}}}

\newcommand{\T}{{\mathbb T}}

\newcommand{\A}{{\mathcal A}}

\newcommand{\D}{{\mathcal D}}
\newcommand{\E}{{\mathbb E}}
\newcommand{\F}{{\mathcal F}}

\renewcommand{\H}{{\mathcal H}}

\newcommand{\R}{{\mathcal R}}

\newcommand{\BMO}{{\mathrm{BMO}}}
\renewcommand{\i}{{\rm i}}

\renewcommand{\a}{\alpha}
\renewcommand{\b}{\beta}
\newcommand{\g}{\gamma}

\renewcommand{\d}{\delta}
\renewcommand{\t}{\theta}
\newcommand{\e}{\varepsilon}
\newcommand{\f}{\varphi}
\newcommand{\p}{\psi}

\renewcommand{\l}{\lambda}
\renewcommand{\O}{\Omega}

\newcommand{\s}{\sigma}
\newcommand{\Si}{\Sigma}

\newcommand{\ot}{\otimes}

\newcommand{\8}{\infty}
\newcommand{\el}{\ell}

\newcommand{\la}{\langle}
\newcommand{\ra}{\rangle}
\newcommand{\wt}{\widetilde}
\newcommand{\wh}{\widehat}
\newcommand{\n}{\noindent}

\newcommand{\les}{\lesssim}
\newcommand{\ges}{\gtrsim }

\newcommand{\be}{\begin{align*}}
\newcommand{\ee}{\end{align*}}
\newcommand{\beq}{\begin{equation}}
\newcommand{\eeq}{\end{equation}}
\newcommand{\beqn}{\begin{equation*}}
\newcommand{\eeqn}{\end{equation*}}

\begin{document}

\title[Optimal orders of the best constants in the Littlewood-Paley inequalities]{Optimal orders of the best constants in the Littlewood-Paley inequalities}


\author[Quanhua  Xu]{Quanhua Xu}
\address{
Laboratoire de Math{\'e}matiques, Universit{\'e} de Bourgogne Franche-Comt{\'e}, 25030 Besan\c{c}on Cedex, France}
\email{qxu@univ-fcomte.fr}

\date{}
\maketitle

 \begin{abstract}
 Let  $\{\mathbb{P}_t\}_{t>0}$ be the classical Poisson semigroup on $\mathbb{R}^d$ and $G^{\mathbb{P}}$ the associated  Littlewood-Paley $g$-function operator:
  $$G^{\mathbb{P}}(f)=\Big(\int_0^\infty t|\frac{\partial}{\partial t} \mathbb{P}_t(f)|^2dt\Big)^{\frac12}.$$
The classical Littlewood-Paley $g$-function inequality asserts that for any $1<p<\infty$ there exist two positive constants $\mathsf{L}^{\mathbb{P}}_{t, p}$ and $\mathsf{L}^{\mathbb{P}}_{c, p}$ such that
 $$
 \big(\mathsf{L}^{\mathbb{P}}_{t, p}\big)^{-1}\big\|f\big\|_{p}\le \big\|G^{\mathbb{P}}(f)\big\|_{p}
 \le  \mathsf{L}^{\mathbb{P}}_{c,p}\big\|f\big\|_{p}\,,\quad f\in L_p(\mathbb{R}^d).
 $$
We determine the optimal orders of magnitude on $p$ of these constants as $p\to1$ and $p\to\infty$. We also consider similar problems for more general test functions in place of the Poisson kernel.

The corresponding problem on the Littlewood-Paley dyadic square function inequality is investigated too. Let $\Delta$ be the partition  of $\mathbb{R}^d$ into dyadic rectangles and  $S_R$ the partial sum operator associated to $R$. The dyadic Littlewood-Paley square function of $f$ is
 $$S^\Delta(f)=\Big(\sum_{R\in\Delta} |S_R(f)|^2\Big)^{\frac12}.$$
For $1<p<\infty$ there exist two positive constants $\mathsf{L}^{\Delta}_{c,p, d}$ and $ \mathsf{L}^{\Delta}_{t,p, d}$ such that
 $$
 \big(\mathsf{L}^{\Delta}_{t,p, d}\big)^{-1}\big\|f\big\|_{p}\le \big\|S^\Delta(f)\big\|_{p}\le \mathsf{L}^{\Delta}_{c,p, d}\big\|f\big\|_{p},\quad f\in L_p(\mathbb{R}^d).
 $$
We show that
 $$\mathsf{L}^{\Delta}_{t,p, d}\approx_d (\mathsf{L}^{\Delta}_{t,p, 1})^d\;\text{ and }\; \mathsf{L}^{\Delta}_{c,p, d}\approx_d (\mathsf{L}^{\Delta}_{c,p, 1})^d.$$

All the previous results can be equally formulated for the $d$-torus $\mathbb{T}^d$. We prove a de Leeuw type transference principle in the vector-valued setting.
 \end{abstract}


\section{Introduction}\label{Introduction}


In the recent investigation \cite{HFC-LPS} on the vector-valued Littlewood-Paley-Stein theory, we are confronted with the problem of determining  the optimal orders of growth on $p$ of the best constants in the classical Littlewood-Paley $g$-function inequality.  The present article deals with this problem as well as the similar one about another classical Littlewood-Paley inequality on the dyadic square function.

\subsection{Littlewood-Paley $g$-function inequality}\label{Littlewood-Paley g-function inequality}

This inequality concerns the $g$-function associated to the Poisson semigroup $\{\mathbb{P}_t\}_{t>0}$ on $\real^d$ whose convolution kernel is
 $$\mathbb{P}_t(x)=\frac{c_d\,t}{(|x|^2+t^2)^{(d+1)/2}}\,.$$
The $g$-square function of  $f\in L_p(\real^d)$ is defined as
  \beq\label{g-function}
  G^{\mathbb{P}}(f)(x)=\Big(\int_0^\8t|\frac{\partial}{\partial t} \mathbb{P}_t(f)(x)|^2dt\Big)^{\frac12}\,, \quad x\in\real^d.
  \eeq
 The classical Littlewood-Paley inequality asserts that for any $1<p<\8$ there exist two positive constants $\mathsf{L}^{\mathbb{P}}_{t, p}$ and $\mathsf{L}^{\mathbb{P}}_{c, p}$ such that
 \beq\label{CLPt}
 \big(\mathsf{L}^{\mathbb{P}}_{t, p}\big)^{-1}\big\|f\big\|_{p}\le \big\|G^{\mathbb{P}}(f)\big\|_{p}
 \le  \mathsf{L}^{\mathbb{P}}_{c,p}\big\|f\big\|_{p}\,,\quad f\in L_p(\real^d).
 \eeq
 We refer to Stein's survey paper \cite{stein-square} for more information as well as historical references. In the sequel, $\mathsf{L}^{\mathbb{P}}_{t, p}$ and $\mathsf{L}^{\mathbb{P}}_{c, p}$ will denote the best constants in \eqref{CLPt}. The use of such notation comes from the vector-valued case studied in \cite{HFC-LPS}, the subscripts $t$ and $c$ refer to type and cotype inequalities, respectively, that is, the complex field $\com$ is of both Lusin type and cotype $2$ in the sense of  \cite{HFC-LPS}. These constants implicitly depend on the dimension $d$ too. However, our main objective concerns the optimal orders of their magnitude on $p$ as $p\to1$ and $p\to\8$ though we also wish to have dimension free estimates.

Curiously, these optimal orders  have not been completely determined in the literature (to the best of our knowledge).
It is also natural to consider the heat semigroup $\{\mathbb{H}_t\}_{t>0}$ on $\real^d$  with kernel
 $$\mathbb{H}_t(x)=(4\pi t)^{-\frac d2}e^{-\frac{|x|^2}{4t}}\,.$$
The associated $g$-function is defined by \eqref{g-function} with $\mathbb{P}$ replaced by $\mathbb{H}$.
Then
 $$
 \big(\mathsf{L}^{\mathbb{H}}_{t, p}\big)^{-1}\big\|f\big\|_{p}\le \big\|G^{\mathbb{H}}(f)\big\|_{p}
 \le  \mathsf{L}^{\mathbb{H}}_{c,p}\big\|f\big\|_{p}\,,\quad f\in L_p(\real^d).
 $$

We will  use the following convention:  $A\lesssim B$ (resp. $A\lesssim_\a B$) means that $A\le C B$ (resp. $A\le C_\a B$) for some absolute positive constant $C$ (resp. a positive constant $C_\a$ depending only on a parameter $\a$).  $A\approx B$ or  $A\approx_\a B$ means that these inequalities as well as their inverses hold.   $p'$ will denote the conjugate index of $p$.

\medskip

The following theorem determines the optimal orders of  the previous constants except  those of  $\mathsf{L}^{\mathbb{P}}_{t, p}$ and  $\mathsf{L}^{\mathbb{H}}_{t, p}$ as $p\to\8$. Part of this  theorem is known, see the historical comments at the end of this subsection.

\begin{thm}\label{PH ML}
 Let $1<p<\8$. Recall that $\mathsf{L}^{\mathbb{P}}_{t, p}$ and  $\mathsf{L}^{\mathbb{P}}_{c, p}$ are the best constants in the following inequalities
  $$\big(\mathsf{L}^{\mathbb{P}}_{t, p}\big)^{-1}\big\|f\big\|_{p}\le \big\|G^{\mathbb{P}}(f)\big\|_{p}
 \le  \mathsf{L}^{\mathbb{P}}_{c,p}\big\|f\big\|_{p}\,,\quad f\in L_p(\real^d).$$
Similarly, we have the best constants $\mathsf{L}^{\mathbb{H}}_{t, p}$ and  $\mathsf{L}^{\mathbb{H}}_{c, p}$ corresponding to the heat semigroup.
Then
 \begin{enumerate}[\rm(i)]
  \item $\mathsf{L}^{\mathbb{P}}_{c,p}\approx \max(\sqrt p,\,p')$ and $\max(\sqrt p,\,p')\les\mathsf{L}^{\mathbb{H}}_{c,p}\les_d \max(\sqrt p,\,p')$;
 \item $1\les \mathsf{L}^{\mathbb{H}}_{t, p}\les \mathsf{L}^{\mathbb{P}}_{t, p}\les_d 1$  for $1<p\le 2$;
 \item $\sqrt p\les\mathsf{L}^{\mathbb{H}}_{t, p}\les \mathsf{L}^{\mathbb{P}}_{t, p}\les p$  for $2\le p<\8$.
 \end{enumerate}
 \end{thm}

We will show a more general result. Given $\e>0$ and $\d>0$ let $\H_{\e, \d}$ denote the class of all functions $\f:\real^d\to\com$ satisfying
\beq\label{Holder}
 \left\{\begin{array}{lcl}
 \displaystyle |\f(x)|\le\frac{1}{(1+|x|)^{d+\e}}, & x\in\real^d;\\
 \displaystyle  |\f(x)-\f(y)|\le \frac{|x-y|^\d}{(1+|x|)^{d+\e+\d}}+\frac{|x-y|^\d}{(1+|y|)^{d+\e+\d}},& x, y\in\real^d;\\
  \displaystyle \int_{\real^d}\f(x)dx=0.
  \end{array}\right.
  \eeq
 We say that $\f$  is {\it nondegenerate} if there exists another function $\p\in \H_{\e, \d}$ such that
 \beq\label{reproduce}
  \int_0^\8 \wh\f(t\xi)\, \wh\p (t\xi)\, \frac{dt}{t} = 1,\quad \forall\xi\in\real^d\setminus\{0\}.
   \eeq
Let $\f_t(x)=\frac1{t^d}\f(\frac{x}{t})$. Define
 \beq\label{G-function}
 G^{\f}(f)(x)=\Big(\int_0^\8|\f_t*f(x)|^2\,\frac{dt}t\Big)^{\frac12}, \quad x\in\real^d
 \eeq
for any (reasonable) function $f$ on $\real^d$. Then it is well known that the following inequality holds
 \beq\label{G-LP}
 \big(\mathsf{L}^{\f}_{t,p}\big)^{-1}\big\|f\big\|_{p}\le \big\|G^{\f}(f)\big\|_{p}\le \mathsf{L}^{\f}_{c,p}\big\|f\big\|_{p},\quad f\in L_p(\real^d).
 \eeq

\begin{thm}\label{fML}
 Let  $\f\in \H_{\e, \d}$ and $1<p<\8$. Then the two best constants $\mathsf{L}^{\f}_{c,p}$ and $\mathsf{L}^{\f}_{t,p}$ in the above inequalities satisfy
 \begin{enumerate}[\rm(i)]
 \item
  $\mathsf{L}^{\f}_{c,p}\les_{d, \e, \d}\max\big(\sqrt p,\, p'\big).$
 \item
  $\mathsf{L}^{\f}_{t,p}\les_{d, \e, \d}\,p$
  if additionally $\f$ is nondegenerate.
  \end{enumerate}
  \end{thm}

Assertion (ii) above implies that  the first inequality of \eqref{G-LP} holds for $p=1$ too. In fact, our proof of Theorem~\ref{fML} implies that the norm $ \big\|G^{\f}(f)\big\|_{L_1(\real^d)}$ is equivalent to the $H_1$-norm of $f$. More precisely, we have the following corollary. Note that there exist several equivalent definitions of  the norm of $H_p(\real^d)$;  the one used in this paper is defined in terms of the $L_p$-norm of the nontangential maximal function of the Poisson integral of a function $f\in H_p(\real^d)$ (cf. \cite{FS}).

\begin{cor}\label{H1}
Let $\f\in \H_{\e, \d}$ be nondegenerate. Then
 $$\big\|G^{\f}(f)\big\|_{1}\approx_{d, \e,\d} \big\|f\big\|_{H_1},\quad f\in H_1(\real^d).$$
Moreover, for any $\p\in  \H_{\e, \d}$ and $1\le p\le 2,\;\a>0$, we have
 $$\max\big\{\big\|G^{\p}(f)\big\|_p,\, \big\|S_\a^{\p}(f)\big\|_p\big\}\les_{d, \e,\d,\a} \big\|G^{\f}(f)\big\|_p,$$
where $S_\a^{\p}(f)$ is the Lusin area integral function:
  \beq\label{S-function}
 S_\a^{\p}(f)(x)=\Big(\int_{|y-x|<\a t}|\p_t*f(x)|^2\,\frac{dy\,dt}{t^{d+1}}\Big)^{\frac12}, \quad x\in\real^d.
 \eeq
\end{cor}

\begin{rk}
 All the previous results hold for the area integral function defined by \eqref{S-function}. In fact, a majority of the literature on the Littlewood-Paley theory deals with  $S^{\f}(f)$ instead of $G^{\f}(f)$.  Namely, if $\f$ is a nondegenerate function in $\H_{\e, \d}$, then for any $1<p<\8$
 $$\big(\mathsf{L}^{\f}_{t,p,S}\big)^{-1}\big\|f\big\|_{p}\le \big\|S_\a^{\f}(f)\big\|_{p}\le \mathsf{L}^{\f}_{c,p, S}\big\|f\big\|_{p},\quad f\in L_p(\real^d).$$
The best constants $\mathsf{L}^{\f}_{c,p, S}$ and $\mathsf{L}^{\f}_{t,p, S}$ satisfy
 \beq\label{S-LP}
 \mathsf{L}^{\f}_{c,p, S}\les_{d, \e, \d,\a}\max\big(\sqrt p,\, p'\big)\;\text{ and }\;\mathsf{L}^{\f}_{t,p, S}\les_{d, \e, \d,\a} \sqrt p\,.
 \eeq
The first one is well known. As far as for the second,  the case $p\le2$ is classical for sufficiently nice $\f$, for instance, for the Poisson kernel. The case $p>2$ is implicitly contained in \cite{CWW}, it can also be found in \cite{Wilson08}  if the aperture $\a$ is large enough, say $\a\ge 3\sqrt d$; for $\a<3\sqrt d$, one can adapt Wilson's argument. See \cite{CDWL, FP, GY} for related results.

Considering the Poisson kernel as in Theorem~\ref{PH ML}, we see that the orders of the constants  in \eqref{S-LP} are optimal as $p\to1$ and $p\to\8$. Thus the optimal orders of the constants  in \eqref{S-LP} are completely determined.
 \end{rk}

 Compared with the area integral function discussed in the above remark, the situation for the $g$-function is more delicate. At the time of this writing, we are unable to determine the optimal orders of $\mathsf{L}^{\mathbb{P}}_{t, p}$ and $\mathsf{L}^{\mathbb{H}}_{t, p}$ as $p\to\8$. The following problem is closely related to the one mentioned on page~239 of \cite{CWW}.

 \begin{problem}\label{pb on LP}
  Let $\f\in \H_{\e, \d}$ be nondegenerate. Does one have $\mathsf{L}^{\f}_{t, p}\les_{d,\e,\d} \sqrt p$ for $2\le p<\8$? In particular, does this hold for the classical Poisson or heat kernel on $\real^d$?
  \end{problem}

 \begin{rk}
  We determined  in \cite{HFC-LPS} the optimal orders of the best constants in \eqref{CLPt} for more general semigroups. Namely, $\{\mathbb{P}_t\}_{t>0}$ in  \eqref{CLPt} can be replaced by the Poisson semigroup $\{P_t\}_{t>0}$ subordinated to any strongly continuous semigroup $\{T_t\}_{t>0}$  of regular contractions on $L_p(\O)$ for a fixed $1<p<\8$. The corresponding constants $\mathsf{L}^{P}_{c, p}$ and $\mathsf{L}^{P}_{t, p}$ also satisfy Theorem~\ref{PH ML}, except assertion (ii) (which is an open problem). We refer to  \cite{HFC-LPS}  for more details. In the work \cite{XZ}, Zhendong Xu and Hao Zhang proved that the optimal order of $\mathsf{L}^{P}_{t, p}$ is $p$ as $p\to\8$ for symmetric markovian semigroups. Thus the previous problem has a negative solution if the classical Poisson or heat semigroup is replaced by a general symmetric markovian semigroup.
 \end{rk}

 It would be also  interesting to have dimension free estimates for $\mathsf{L}^{\mathbb{H}}_{c,p}$ in Theorem~\ref{PH ML} (i):

 \begin{problem}
 Does one have $\mathsf{L}^{\mathbb{H}}_{c,p}\les \max(\sqrt p,\,p')$?
  \end{problem}

Meyer \cite{Meyer} shows that this question has an affirmative answer for $p>2$ if the time derivative in the definition of the $g$-function $G^{\mathbb{H}}$ is replaced  by the spatial derivative. More precisely, let
 $$G^{\mathbb{H}}_{\nabla_x}(f)(x)=\Big(\int_0^\8|\nabla_x\mathbb{H}_t(f)(x)|^2dt\Big)^{\frac12}\,, $$
where $\nabla_x=(\frac{\partial}{\partial x_1},\cdots, \frac{\partial}{\partial x_d})$. Then
 $$\big\|G^{\mathbb{H}}_{\nabla_x}(f)\big\|_p\les \sqrt p\, \big\|f\big\|_p,\quad f\in L_p(\real^d), \;2\le p<\8.$$

\begin{rk}\label{RT}
 All the previous results admit periodic analogues, that is, for $\T^d$ in place of $\real^d$, where $\T^d$ is the $d$-torus equipped with normalized Haar measure. This can be done by modifying the arguments for $\real^d$, or more simply by a de Leeuw type transference principle. We will show a variant of de Leeuw's theorem in section~\ref{A variant of de Leeuw's multiplier theorem}. Using this de Leeuw type theorem, as illustration we will explain in section~\ref{Proof of Theorem PH ML} why the constant $\mathsf{L}^{\mathbb P}_{c, p}$ and its period analogue are essentially the same. A full argument can be found  in the proof of Theorem~\ref{RvsT} below.
\end{rk}

\medskip\n{\bf Historical comments.} We make  some remarks on the previous results and especially point out the part of them which are  known or implicit in the literature.

 \begin{enumerate}[(i)]
 \item The classical Littlewood-Paley $g$-function is usually defined by using the full gradient of $\mathbb{P}_t(f)$ in place of the partial derivative in the time variable. Let us denote the latter by $G_\nabla^{\mathbb{P}}(f)$:
 \beq\label{gradient G}
 G_\nabla^{\mathbb{P}}(f)=\Big(\int_0^\8t|\nabla\mathbb{P}_t(f)|^2dt\Big)^{\frac12}\,.
 \eeq
This square function behaves better than the previous one since it is invariant under the Riesz transforms. Theorem~\ref{PH ML} equally holds for $G_\nabla^{\mathbb{P}}$ in place of $G^{\mathbb{P}}$. The corresponding proof is slightly simpler (see the related remark on page~ 172 of \cite{FS}).

\item Part (i) of Theorem~\ref{fML} is known and can be found in Wilson's book \cite{Wilson08}. In fact, Wilson shows a stronger result on his intrinsic $g$-function that is defined by
 $$G_{\e, \d}(f)(x)=\Big(\int_{0}^\8\,\sup_{\f\in\H_{\e, \d}} |\f_t*f(x)|^2\,\frac{dt}{t}\Big)^{\frac12}.$$
Then for any weight $w$ on $\real^d$ and $f\in L_2(\real^d)$,
 $$\int_{\real^d} G_{\e, \d}(f)^2w \les_{d, \e, \d} \int_{\real^d} |f|^2M(w),$$
 where $M(w)$ denotes the Hardy-Littlewood maximal function of $w$. This implies that for $2\le p<\8$
 $$\big\|G_{\e, \d}(f)\big\|_{p} \les_{d, \e, \d}\sqrt{p}\,\big\|f\big\|_{p},$$
whence $\mathsf{L}^{\f}_{c,p}\les_{d, \e, \d}\sqrt p$ for $p\ge 2$. The case $p<2$ is dealt with by a standard argument involving singular integral theory for $G^\f$ can be expressed as a Calder\'on-Zygmund operator with Hilbert space valued kernel (see Lemma~\ref{CZ} below).

\item  Part (ii) of Theorem~\ref{fML} for $p>2$ immediately follows from  part (i) by duality.

\item The upper estimate $\big\|G^{\f}(f)\big\|_{L_1(\real^d)}\les_{d, \e,\d} \big\|f\big\|_{H_1(\real^d)}$ in Corollary~\ref{H1} is also well known for the same reason via singular integral theory. The converse inequality is classical too for the Poisson kernel.

\item Wilson's theorem just quoted implies particularly $\mathsf{L}^{\mathbb{P}}_{c,p}\les_{d}\sqrt p$ for $p\ge 2$. However, the dimension free estimate $\mathsf{L}^{\mathbb{P}}_{c,p}\les\sqrt p$
in Theorem~\ref{PH ML} (i) is due to Meyer \cite{Meyer} by probabilistic method. In \cite{HFC-LPS} we prove a similar result for general semigroups. The dimension free estimate $\mathsf{L}^{\mathbb{P}}_{t,p}\les p$ for $p\ge 2$ in Theorem~\ref{PH ML} (iii) is also a special case of that result of \cite{HFC-LPS}.

\item We can now summarize the new part of Theorems~\ref{PH ML}-\ref{fML} and Corollary~\ref{H1} as follows:
\begin{align*}
&\mathsf{L}^{\mathbb{H}}_{c,p}\ges\mathsf{L}^{\mathbb{P}}_{c,p}\ges \sqrt p\;\text{ and }\; \mathsf{L}^{\mathbb{P}}_{t, p}\ges \mathsf{L}^{\mathbb{H}}_{t,p}\ges \sqrt p\;\text{ for }\; 2\le p<\8;\\
&\mathsf{L}^{\mathbb{P}}_{c,p}\les p' \;\text{ for }\; 1< p\le2\;\text{ and }\; \mathsf{L}^{\mathbb{P}}_{t,p}\les p \;\text{ for }\; p\ge 2 \;\; (\text{dimension freeness)};\\
&\big\|f\big\|_{H_1}\les_{d, \e,\d} \big\|G^{\f}(f)\big\|_1 \;\text{ for any nondegenerate }\;\f\in\H_{\e,\d}.
 \end{align*}
 \end{enumerate}

\medskip

\subsection{Littlewood-Paley dyadic square function inequality}

There exists another equally famous inequality named after Littlewood-Paley, the one related to the dyadic decomposition of $\real^d$. First, partition $\real\setminus\{0\}$ into the intervals $[2^{k-1},\, 2^{k})$ and $(-2^{k},\, -2^{k-1}]$, $k\in\ent$. Then the family $\Delta$ of all $d$-fold products of these intervals give a partition of $\real^d$ (deprived of the origin). For any $R\in\Delta$ let $S_R$ be the corresponding partial sum operator, that is, $\wh{S_R(f)}=\un_{R}\wh f$. The dyadic  Littlewood-Paley square function of $f$ is defined by
 $$S^\Delta(f)=\Big(\sum_{R\in\Delta} |S_R(f)|^2\Big)^{\frac12}.$$
It is well known that for $1<p<\8$ there exist two positive constants $\mathsf{L}^{\Delta}_{c,p, d}$ and $ \mathsf{L}^{\Delta}_{t,p, d}$ such that
\beq\label{DLP-R}
 \big(\mathsf{L}^{\Delta}_{t,p, d}\big)^{-1}\big\|f\big\|_{p}\le \big\|S^\Delta(f)\big\|_{p}\le \mathsf{L}^{\Delta}_{c,p, d}\big\|f\big\|_{p},\quad f\in L_p(\real^d).
 \eeq
See, for instance, \cite[Theorem~5.1.6]{Grafakos} where it is showed that both constants $\mathsf{L}^{\Delta}_{c,p, d}$ and $\mathsf{L}^{\Delta}_{t,p, d}$ are majorized by $C_d \max(p, p')^{2d}$. However, like for the $g$-function inequality, their optimal orders have not been completely determined in the literature. Note that the Littlewood-Paley $g$-function inequality belongs to one-parameter harmonic analysis while the above inequality is of multi-parameter nature. This explains why we now mention $d$ explicitly as a subscript in the above constants.

\smallskip

In the spirit of Remark~\ref{RT}, we formulate the periodic counterpart of \eqref{DLP-R}. The only difference is that the dyadic rectangles now consist of integers, so the corresponding dyadic partition of $\ent^d$ is $\wt\Delta=\{R\cap\ent^d\,:\, R\in\Delta\}$.  We similarly define the partial sum operators $S_R$ and
 $$S^{\wt\Delta}(f)=\Big(\sum_{R\in\wt\Delta} |S_R(f)|^2\Big)^{\frac12},\quad f\in L_p(\T^d).$$
So \eqref{DLP-R} becomes
\beq\label{DLP-T}
 \big(\mathsf{L}^{\wt\Delta}_{t,p, d}\big)^{-1}\big\|f\big\|_{p}\le \big\|S^{\wt\Delta}(f)\big\|_{p}\le \mathsf{L}^{\wt\Delta}_{c,p, d}\big\|f\big\|_{p},\quad f\in L_p(\T^d).
 \eeq

The following reinforces the meaning of Remark~\ref{RT}.

\begin{thm}\label{RvsT}
 Let $1<p<\8$. Then the best constants in \eqref{DLP-R} and \eqref{DLP-T} satisfy
 $$\mathsf{L}^{\Delta}_{c,p, d}=\mathsf{L}^{\wt\Delta}_{c,p, d}\quad\text{ and }\quad \mathsf{L}^{\Delta}_{t,p, d}=\mathsf{L}^{\wt\Delta}_{t,p, d}.$$
\end{thm}

\smallskip

The following result illustrates the multi-parameter nature of \eqref{DLP-R} and  \eqref{DLP-T}. In view of the previous theorem, we need to state it only for the periodic case.

\begin{thm}\label{multi-parameter}
 There exists a universal positive constant $C$ such that
 $$\big(\mathsf{L}^{\wt\Delta}_{c,p, 1}\big)^d\le \mathsf{L}^{\wt\Delta}_{c,p, d}\le \big(C\mathsf{L}^{\wt\Delta}_{c,p, 1}\big)^d\quad\text{ and }\quad
 \big(\mathsf{L}^{\wt\Delta}_{t,p, 1}\big)^d\le \mathsf{L}^{\wt\Delta}_{t,p, d}\le \big(C\mathsf{L}^{\wt\Delta}_{t,p, 1}\big)^d.$$
\end{thm}

The following corollary determines the optimal orders of all the best constants (except one) in \eqref{DLP-R}, so as well as of those in \eqref{DLP-T}.

\begin{cor}\label{DLP-const}
 Let $1<p<\8$. Then
 \begin{enumerate}[\rm(i)]
  \item $\mathsf{L}^{\Delta}_{c,p, d}\approx_d p'^{\frac{3}2d}$ for  $1<p\le2$ and $\mathsf{L}^{\Delta}_{c,p, d}\approx_d p^{d}$ for  $2<p<\8$;
 \item $\mathsf{L}^{\Delta}_{t,p, d}\approx_d 1$ for  $1<p\le2$ and
$p^{\frac{d}2}\les\mathsf{L}^{\Delta}_{t,p, d}\les_d p^d$ for $2\le p<\8$.
 \end{enumerate}
 \end{cor}

\begin{rk}
 The estimate $\mathsf{L}^{\Delta}_{t,p, d}\les_d p^d$ for $2\le p<\8$ was proved independently by Odysseas Bakas and Hao Zhang after the submission of this article; it improves the author's original one $\mathsf{L}^{\Delta}_{t,p, d}\les_d (p\log p)^d$. However, like for the $g$-function inequality, we are unable to determine the optimal order of $\mathsf{L}^{\Delta}_{t,p, d}$ for $p>2$.  We need to do this only for $d=1$ by virtue of Theorem~\ref{multi-parameter}.
\end{rk}

 \begin{problem}
  Determine the optimal order of $\mathsf{L}^{\Delta}_{t,p, 1}$ as $p\to\8$.
  \end{problem}

 This problem is related to Problem~\ref{pb on LP}. In fact, the smooth version of the dyadic square function $S^{\Delta}$ is a discrete $g$-function, so the analogue for the smooth version of the above problem is a particular case of Problem~\ref{pb on LP}.

\medskip\n{\bf Historical comments.} Part of Corollary~\ref{DLP-const}  is already known.

 \begin{enumerate}[(i)]
 \item It is Bourgain \cite{Bourgain89} who first studied the problem on the optimal orders of the above constants by determining the optimal order of $\mathsf{L}^{\wt\Delta}_{c,p, 1}$.
Lerner \cite{Lerner19}  noted that Bourgain's result remains valid for $\real$ by a different method via weighted norm inequalities.
 \item Bourgain \cite{Bourgain85} proved that $\mathsf{L}^{\wt\Delta}_{t,p, 1}\approx1$. In fact, Bourgain showed that the second inequality of \eqref{DLP-T} holds for any partition of $\ent$ into bounded intervals in the case of $1\le p\le2$. This latter result is dual to Rubio de Francia's celebrated Littlewood-Paley inequality \cite{rubio} that insures the validity of the first inequality of \eqref{DLP-T} for any partition of $\ent$ into bounded intervals in the case of $2\le p<\8$.
  \item Bakas \cite{Bakas19} extended one of Bourgain's estimates to the higher dimensions by showing
 $\mathsf{L}^{\wt\Delta}_{c,p, d}\approx_d p'^{\frac{3d}2}$ for $1<p\le2$.
 \item Journ\'e \cite{journe} extended  Rubio de Francia's inequality to the multi-dimensional setting without explicit estimate of the relevant constant. It is Osipov who proved  the second inequality of \eqref{DLP-R} for $1\le p\le2$ and for any partition of $\real^d$ into bounded rectangles. In particular, Osipov's result implies
 $\mathsf{L}^{\Delta}_{t,p, d}\approx_d1$ for $1<p\le2.$
 \item The estimate $p^{\frac{d}2}\les\mathsf{L}^{\Delta}_{t,p, d}$ for $2\le p<\8$ easily follows from the optimal order of the best constant in the Khintchine inequality for $p>2$; on the other hand,  Pichorides \cite{Pichorides90} proved $\mathsf{L}^{\wt\Delta}_{t,p, 1}\les p\log p$ for $2\le p<\8$.
 \end{enumerate}


\section{A variant of de Leeuw's multiplier theorem}\label{A variant of de Leeuw's multiplier theorem}


In this section we give a variant of de Leeuw's classical transference theorem on Fourier multipliers on $\real^d$ and $\T^d$, see \cite{Leeuw} and \cite[Chapter VII.3]{Stein-Weiss}.

We begin by fixing some notation. Given $z=(z_1,\cdots, z_d)\in\T^d$ and $m=(m_1,\cdots, m_d)\in\ent^d$ let $z^m=(z_1^{m_1}, \cdots, z_d^{m_d})$. We identify $\T^d$ with the cube $\mathbb I^d=[-\frac12,\,\frac12)^d\subset\real^d$ via
$z=(e^{2 \pi \mathrm{i} x_1},\cdots , e^{2 \pi \mathrm{i} x_d})\;\leftrightarrow\; x=(x_1, \cdots, x_d)$, and accordingly the functions on $\T^d$ with the $1$-periodic functions on $\real^d$.

Let $X, Y$  be two Banach spaces and $B(X, Y)$ the space of continuous linear operators from $X$ to $Y$. Given a function $\f: \real^d\to B(X, Y)$, let $T_\f$  be the Fourier multiplier formally defined by $\wh{T_\f(f)}(\xi)=\f(\xi)\wh{f}(\xi)$ for $\xi\in\real^d$ and $f\in L_p(\real^d; X)$; similarly,  define the Fourier multiplier $M_\f$ in the periodic case when $\f$ is restricted to $\ent^d$, namely,  $\wh{M_\f(f)}(m)=\f(m)\wh{f}(m)$ for $m\in\ent^d$ and  $f\in L_p(\T^d; X)$. Here given a measure space $(\O,\mu)$, $L_p(\O; X)$ denote the space of $p$-integrable functions from $\O$ to $X$.

\smallskip

We will assume that  the symbol $\f$ satisfies the following conditions:
\begin{itemize}
  \item[({\bf H}$_1$)] for every $a\in X$,  $\f(\cdot)(a)$ is a measurable function from $\real^d$ to $Y$, and $\f$ is bounded, i.e., $M=\sup_{\xi\in\real^d}\|\f(\xi)\|_{B(X, Y)}<\8$;
 \item[({\bf H}$_2$)] there exists a partition $\R$ of $\real^d$ into bounded rectangles such that $\f$ is strongly continuous on every  $R\in\R$, i.e.,  $\f(\cdot)(a)$ is continuous from $R$ to $Y$ for every $a\in X$;
  \item[({\bf H}$_3$)] for every $a\in X$ and every $R\in\R$ the range of the restriction of $\f(\cdot)(a)$ to $R$ is contained in a finite dimensional subspace of $Y$;
 \item[({\bf H}$_4$)] for every $1<p<\8$ and every $a\in X$ there exists a constant $C_{p,a}$ such that
 $$\big\| T_\f(a f)\big\|_{L_p(\real^d; Y)}\le C_{p,a} \|f\|_{L_p(\real^d)}$$
for all compactly supported $C^\8$ functions $f$ on $\real^d$.
\end{itemize}

The last condition implies that $T_\f(f)$ is well-defined and belongs to $L_p(\real^d; Y)$ for any compactly supported $C^\8$ function $f$ with values in a finite dimensional subspace of $X$.

\smallskip

We are interested in the best constants $\a$ and $\b$ in the following inequalities
 \beq\label{Fourier R}
 \a^{-1}\|f\|_{L_p(\real^d; X)}\le \| T_\f(f)\|_{L_p(\real^d; Y)}\le \b\|f\|_{L_p(\real^d; X)}
 \eeq
for all compactly supported $C^\8$ functions $f$ on $\real^d$ with values in a finite dimensional subspace of $X$. Obviously, if it is finite, $\b$ is equal to the norm of $T_\f$ as a map from $L_p(\real^d; X)$ to $L_p(\real^d; Y)$; we will denote this norm simply by $\|T_\f\|_{p\to p}$. Similarly and by a slight abuse of notation, if $\a$ is finite, we will denote it by $\|T_\f^{-1}\|_{p\to p}$ which is the norm of $T_\f^{-1}$ from the image of $T_\f$ in $L_p(\real^d; Y)$ to $L_p(\real^d; X)$.

We will also consider the periodic version of \eqref{Fourier R}, the corresponding constants will be denoted by $\|M_\f\|_{p\to p}$ and $\|M_\f^{-1}\|_{p\to p}$ which are the best constants such that
 \beq\label{Fourier T}
\|M_\f^{-1}\|_{p\to p}^{-1}\,\|f\|_{L_p(\T^d; X)}\le \|M_\f(f)\|_{L_p(\real^d; Y)}\le\|M_\f\|_{p\to p}\,\|f\|_{L_p(\T^d; X)}
 \eeq
for all  trigonometric polynomials $f$ with coefficients in $X$.

We make the convention that if one of the inequalities in \eqref{Fourier R} and \eqref{Fourier T} does not hold, the corresponding constant is understood to be infinite.

\begin{thm}\label{de Leeuw}
 Let $1<p<\8$.
 \begin{enumerate}[\rm(i)]
 \item Assume that $\f$ is strongly continuous at every $m\in\ent^d$. Then
  $$\|M_\f\|_{p\to p}\le \|T_\f\|_{p\to p}\quad \text { and }\quad \|M_\f^{-1}\|_{p\to p}\le\|T_\f^{-1}\|_{p\to p}.$$
 \item Given $t>0$ define $\f^{(t)}$ by $\f^{(t)}(\xi)=\f(t\xi)$. Then
 $$\|T_\f\|_{p\to p}\le \liminf_{t\to0}\|M_{\f^{(t)}}\|_{p\to p}\quad\text { and }\quad \|T_\f^{-1}\|_{p\to p}\le \liminf_{t\to0}\|M^{-1}_{\f^{(t)}}\|_{p\to p}.$$
\end{enumerate}
\end{thm}

We will adapt de Leeuw's arguments. Note, however, that de Leeuw's proof depends on a duality argument that does not seem to extend to our setting. Instead, we establish a direct link between $T_\f$ and $M_\f$ as in Lemma~\ref{link1} below.

\smallskip

The following lemma is a well known elementary fact (see \cite[Lemma~VII.3.9]{Stein-Weiss}).

\begin{lem}\label{ergodic}
Let $f\in L_1(\T^d)$. Then
$$\lim_{t\to\8}\int_{\real^d}f(x)\mathbb{H}_t(x)dx=\int_{\T^d}f(z)dz.$$
\end{lem}

The following  expresses the periodic Fourier multiplier $M_\f$ in terms of the Euclidean $T_\f$.

\begin{lem}\label{link1}
Assume  that $\f$ is strongly continuous at every point $m\in\ent^d$. Let $P$ be a trigonometric polynomial with coefficients in $X$. Then
$$\lim_{t\to0}(4\pi pt)^{\frac{d}{2p}}\big\|T_\f(P\,\wh{\mathbb{H}_t})\big\|_{L_p(\real^d; Y)}=\big\|M_\f(P)\big\|_{L_p(\T^d; Y)}.$$
\end{lem}

\begin{proof}
 By approximation, we can assume that $\f$ is compactly supported. Let
 $$P(z)=\sum_m a_m z^m,\quad a_m\in X.$$
Then
 $$\wh{P\,\wh{\mathbb{H}_t}}(\xi)=\sum_m a_m \mathbb{H}_t(\xi-m).$$
Thus
 \begin{align*}
 T_\f(P\,\wh{\mathbb{H}_t})(x)
 &=\sum_m  \int_{\real^d}\f(\xi)(a_m) \mathbb{H}_t(\xi-m)e^{2\pi\i \xi\cdot x}d\xi\\
 &=M_\f(P)(x)\wh{\mathbb{H}_t}(x)
   +\sum_m \int_{\real^d}(\f(\xi)-\f(m))(a_m) \mathbb{H}_t(\xi-m)e^{2\pi\i \xi\cdot x}d\xi\\
   &{\mathop=^{\rm def}}M_\f(P)(x)\wh{\mathbb{H}_t}(x)+\sum_m f_{m,t}(x).
 \end{align*}
Recall that
 $\wh{\mathbb{H}_t}(x)=e^{-4\pi^2 t|x|^2}$.
Letting $s=(16\pi^2pt)^{-1}$ and using Lemma~\ref{ergodic}, we get
  \begin{align*}
  \lim_{t\to0}(4\pi pt)^{\frac{d}2}\int_{\real^d}\big\|M_\f(P)(x)\wh{\mathbb{H}_t}(x)\big\|^p_Ydx
  =\lim_{s\to\8}\int_{\real^d}\big\|M_\f(P)(x)\big\|^p_Y \mathbb{H}_{s}(x)dx
  =\big\|M_\f(P)\big\|_{L_p(\T^d; Y)}^p.
   \end{align*}
Thus it remains to show that
  \beq\label{zero term}
  \lim_{t\to0}(4\pi pt)^{\frac{d}{2p}}\big\|f_{m,t}\big\|_{L_p(\real^d; Y)}=0,\quad \forall\, m.
  \eeq
Choose $q\in(1, \,\8)$ and $\t\in(0,\,1)$ such that $\frac1{p}=\frac{1-\t}{2}+\frac\t{q}$. Then
 $$(4\pi pt)^{\frac{d}{2p}}\big\|f_{m,t}\big\|_{L_p(\real^d; Y)}\le \big[(4\pi pt)^{\frac{d}{4}}\big\|f_{m,t}\big\|_{L_2(\real^d; Y)}\big]^{1-\t}\,
 \big[(4\pi pt)^{\frac{d}{2q}}\big\|f_{m,t}\big\|_{L_q(\real^d; Y)}\big]^\t.$$
By  ({\bf H}$_1$) and ({\bf H}$_4$),
 $$\big\|f_{m,t}\big\|_{L_q(\real^d; Y)}\le (C_{q, a_m} +M\|a_m\|_X)\big\| z^m \,\wh{\mathbb{H}_t}\big\|_{L_q(\real^d)}= (C_{q, a_m} +M\|a_m\|_X) \big\|\wh{\mathbb{H}_t}\big\|_{L_q(\real^d)}.$$
 It follows that
 $$\sup_{t>0}(4\pi pt)^{\frac{d}{2q}}\big\|f_{m,t}\big\|_{L_q(\real^d; Y)}<\8.$$
Let us treat the part on the $L_2$-norm. Since we have assumed that $\f$ is compactly supported, by  ({\bf H}$_3$), we can further assume that $Y$ is  finite dimensional, so isomorphic to a Hilbert space. Thus by the Plancherel identity, there exists a constant $C$, depending on $a_m$ and $Y$, such that
 $$\big\|f_{m,t}\big\|_{L_2(\real^d; Y)}^2\le C^2\int_{\real^d}\big\|(\f(\xi)-\f(m))(a_m)\big\|_Y^2 \mathbb{H}_t(\xi-m)^2d\xi.$$
Given $\e>0$,  the strong continuity of $\f$ at $m$ implies that there exists $\d>0$ such that
 $\big\|(\f(\xi)-\f(m))(a_m)\big\|_Y<\e$ whenever $|\xi-m|<\d$.
Thus by ({\bf H}$_1$),
  \begin{align*}
  &(4\pi pt)^{\frac{d}{2}}\int_{\real^d}\big\|(\f(\xi)-\f(m))(a_m)\big\|_Y^2 \mathbb{H}_t(\xi-m)^2d\xi\\
 &\hskip .2cm \le \e^2 (4\pi pt)^{\frac{d}{2}}\int_{|\xi-m|<\d} \mathbb{H}_{t}(\xi-m)^2d\xi + (2M\|a_m\|_X)^2 (4\pi pt)^{\frac{d}{2}} \int_{|\xi-m|\ge\d} \mathbb{H}_{t}(\xi-m)^2d\xi\\
 & \les_{p, d} \e^2 + (2M\|a_m\|_X)^2  \int_{|\xi|\ge\frac{\d}{\sqrt{2t}}} e^{-|\xi|^2}d\xi.
  \end{align*}
 Therefore,
  $$\limsup_{t\to0}(4\pi pt)^{\frac{d}{4}}\big\|f_{m,t}\big\|_{L_2(\real^d; Y)}\les_{p, d}C\e.$$
 As $\e$ is arbitrary, combining the above estimates, we deduce \eqref{zero term}.
\end{proof}

Conversely, we can estimate $T_\f$ in terms of $M_\f$.

\begin{lem}\label{link2}
 Let $f:\real^d\to X$ be a $C^\8$ function with compact support and define the periodization of $\wt{f_t}$:
 $$\wt{f_t}(x)=\sum_{m\in\ent^d}f_t(x+m), \quad x\in\real^d.$$
Viewing $\wt{f_t}$ as a function on $\T^d$, we have
 $$\lim_{t\to0}t^{\frac{d}{p'}}\big\|M_{\f^{(t)}}(\wt{f_t})\big\|_{L_p(\T^d; Y)}=\big\|T_{\f}(f)\big\|_{L_p(\real^d; Y)}.$$
\end{lem}

\begin{proof}
 The Fourier series of $\wt{f_t}$ is given by
 $$\wt{f_t}(x)=\sum_{m\in\ent^d}\wh{f}(tm)e^{2\pi\i m\cdot x}.$$
Thus
   \begin{align*}
   \lim_{t\to0} t^dM_{\f^{(t)}}(\wt{f_t})(tx)
   &=\lim_{t\to0} t^d \sum_{m\in\ent^d}\f(tm)\wh{f}(tm)e^{2\pi\i tm\cdot x}\\
 &=\int_{\real^d}\f(\xi)\wh f(\xi)e^{2\pi\i \xi\cdot x}d\xi=T_\f(f)(x),
   \end{align*}
where we have used ({\bf H}$_2$) to insure that the above integral exists in Riemann's sense. Let $\eta$ be a nonnegative continuous function with compact support on $\real^d$ such that
 $$\eta(0)=1\quad\text{ and }\quad \sum_{m\in\ent^d}\eta(x+m)^p=1$$
(see  \cite[Lemma~VII.3.21]{Stein-Weiss}). Then
 $$ \lim_{t\to0} t^dM_{\f^{(t)}}(\wt{f_t})(tx)\eta(tx)=T_\f(f)(x).$$
Thus
 $$\big\|T_{\f}(f)\big\|_{L_p(\real^d; Y)}=\lim_{t\to0} t^d \big\|M_{\f^{(t)}}(\wt{f_t})(t\cdot)\eta(t\cdot)\big\|_{L_p(\real^d; Y)}.$$
 However,
  \begin{align*}
 \big\|M_{\f^{(t)}}(\wt{f_t})(t\cdot)\eta(t\cdot)\big\|_{L_p(\real^d; Y)}^p
  &=t^{-d}\int_{\real^d}\|M_{\f^{(t)}}(\wt{f_t})(x)\|_Y^p\eta(x)^pdx\\
  &=t^{-d}\sum_{m\in\ent^d}\int_{\mathbb I^d}\|M_{\f^{(t)}}(\wt{f_t})(x)\|_Y^p\eta(x+m)^pdx\\
  &=t^{-d}\big\|M_{\f^{(t)}}(\wt{f_t})\big\|_{L_p(\T^d; Y)}^p.
   \end{align*}
 We then deduce the desired assertion.
 \end{proof}

 \begin{proof}[Proof of Theorem~\ref{de Leeuw}]
  (i) Let $P$ be a trigonometric polynomial with coefficients in $X$. By Lemma~\ref{link1} and Lemma~\ref{ergodic}
    \begin{align*}
   \big\|M_\f(P)\big\|_{L_p(\T^d; Y)}
   &\le \|T_\f\|_{p\to p}\, \lim_{t\to0}(4\pi pt)^{\frac{d}{2p}}\big\|P\,\wh{\mathbb{H}_t}\big\|_{L_p(\real^d; Y)}\\
   &=\|T_\f\|_{p\to p}\, \big\|P\big\|_{L_p(\T^d; Y)},
  \end{align*}
 whence $\|M_\f\|_{p\to p}\le \|T_\f\|_{p\to p}$. The second inequality $\|M^{-1}_\f\|_{p\to p}\le \|T^{-1}_\f\|_{p\to p}$ is proved in the same way.

 (ii) We use Lemma~\ref{link2} for this part. Let $f$ be a compactly supported function with values in a finite dimensional subspace of $X$. Then for $t$ sufficiently small, $f_t$ is supported in the cube $\mathbb I^d$, so $\wt{f_t}=f_t$. Thus
   \begin{align*}
  \big\|M_{\f^{(t)}}(\wt{f_t})\big\|_{L_p(\T^d; Y)}
  &\le \|M_{\f^{(t)}}\|_{p\to p}\, \big\|\wt{f_t}\big\|_{L_p(\T^d; Y)}\\
  &=\|M_{\f^{(t)}}\|_{p\to p}\,\big\|f_t\big\|_{L_p(\real^d; Y)}\\
  &=t^{-\frac{d}{p'}}\|M_{\f^{(t)}}\|_{p\to p}\,\big\|f\big\|_{L_p(\real^d; Y)}\,.
  \end{align*}
  Therefore, by Lemma~\ref{link2} ,
   $$\big\|T_{\f}(f)\big\|_{L_p(\real^d; Y)}\le \liminf_{t\to0} \|M_{\f^{(t)}}\|_{p\to p}\,\big\|f\big\|_{L_p(\real^d; Y)},$$
  whence
   $$\|T_\f\|_{p\to p}\le \liminf_{t\to0} \|M_{\f^{(t)}}\|_{p\to p}.$$
   We show similarly the other inequality of part (ii).
    \end{proof}

 \begin{rk}\label{de Leeuw bis}
  Some of the hypotheses  ({\bf H}$_1$)--({\bf H}$_4$) can be weakened for the validity of Theorem~\ref{de Leeuw}. We have seen in the proof of Lemma~\ref{link2} that  ({\bf H}$_2$) can be replaced by the  Riemann integrability of the function $\xi\mapsto \f(\xi)\wh f(\xi)e^{2\pi\i \xi\cdot x}$  for every compactly supported $C^\8$ function $f$. On the other hand, the proof of Lemma~\ref{link1} shows that ({\bf H}$_3$) is unnecessary if $Y$ is isomorphic to  a Hilbert space.
 \end{rk}

If the function $\f$ is strongly continuous, Theorem~\ref{de Leeuw} can be reformulated as follows: the norms  $\|T_\f\|_{p\to p}$ and $\|T^{-1}_\f\|_{p\to p}$ coincide with the corresponding ones when $\real^d$ is viewed as a discrete group (see \cite{Leeuw} for more details). We then obtain the following corollary as in \cite{Leeuw}:

\begin{cor}\label{subgroup}
 Let $\f:\real^d\to B(X, Y)$ be a strongly continuous function. Let $\psi$ be the restriction of $\f$ to $\real^k\subset \real^d$ for some $k<d$. Consider the Fourier multiplier $T_\psi$ from $L_p(\real^k; X)$ to $L_p(\real^k; Y)$. Then
 $$\|T_\psi\|_{p\to p}\le \|T_\f\|_{p\to p}\;\;\text{ and }\;\; \|T^{-1}_\psi\|_{p\to p}\le \|T^{-1}_\f\|_{p\to p}\,.$$
\end{cor}


\section{Proofs of Theorem~\ref{fML} and Corollary~\ref{H1}}


As described in the historical comments of subsection~\ref{Littlewood-Paley g-function inequality},  we need only to show Theorem~\ref{fML} (ii) for $p<2$. We start the proof by some preliminaries.  In the sequel, $Q$ will denote a cube of $\real^d$ (with sides parallel to the axes), $|Q|$ and $\el(Q)$ being respectively its volume and side length.
For a locally integrable function $f$ on $\real^d$  we let $\la f \ra_Q$ denote the mean of $f$ over $Q$:
 $$\la f \ra_Q= \frac{1}{|Q|} \int_Q f(x)dx.$$

 As mentioned before, part (i) of Theorem~\ref{fML} for $p\le 2$ is proved by singular integrals. We state this result as a lemma for later use.

\begin{lem}\label{CZ}
 Let $\f\in\H_{\e,\d}$ and $f\in H_p(\real^d)$ with $1\le p\le2$. Then
 $$\|G^\f(f)\|_p\les_{d,\e,\d} \|f\|_{H_p}.$$
\end{lem}

Indeed, consider  the Hilbert space valued kernel  $K$ defined by $K(x)=\{\f_t(x)\}_{t>0}$ for $x\in\real^d$, that is, $K$ is a function from $\real^d$ to $L_2(\real_+)$, where $\real_+$ is equipped with the measure $\frac{dt}t$.  We  use $K$ to denote the associated singular integral too:
   $$K(f)=\int_\real K(x-y)f(y)dy.$$
 Then
 $$
G^{\f}(f)(x)=\big\|K(f)(x)\big\|_{L_2(\real_+)}\,,\quad x\in\real^d.
 $$
 It is easy to show that  $K$ satisfies the  following regularities (see below for a proof):
 \beq\label{regularity}
 \big\|K(x)\big\|_{L_2(\real_+)} \les_{\e} \frac{1}{|x|^d}\;\text{ and }\; \big\|K(x+z)-K(x)\big\|_{L_2(\real_+)} \les_{\e,\d} \frac{|z|^\d}{|x|^{d+\d}},\quad x, z\in\real^d,\; |x|>2|z|.
 \eeq
 Thus the lemma follows from the $L_2$-boundedness of $K$ and the Calder\'on-Zygmund theory.

We will need a reinforcement of the previous lemma for Wilson's intrinsic square functions defined by
\begin{align*}
  S_{\e, \d}(f)(x)
  =\Big(\int_{|y-x|<t}\,\sup_{\f\in\H_{\e, \d}}\big|\f_t*f(y)\big|^2\,\frac{dy\,dt}{t^{d+1}}\Big)^{\frac12}.
    \end{align*}
 This square function can also be expressed as a singular integral operator. Let the cone $\Gamma=\{(y, t)\in\real^{d+1}_+: |y|<t\}$ be equipped with the measure $\frac{dydt}{t^{d+1}}$. Let $X$ be the Banach space of square integrable functions on $\Gamma$ with values in $\el_\8(\H_{\e, \d})$:
 $$X=L_2\big(\Gamma; \el_\8(\H_{\e, \d})\big).$$
This time, the convolution kernel $K$ is an $X$-valued kernel: for $x\in\real^d$, $K(x)$ is defined as follows:
 $$K(x): \Gamma\to \el_\8(\H_{\e, \d}),\quad (y, t)\mapsto \{\f_t(x+y)\}_{\f\in\H_{\e, \d}}.$$
Then
 $$
S_{\e, \d}(f)(x)=\big\|K(f)(x)\big\|_{X}\,,\quad x\in\real^d.
 $$
Let us show that this new kernel $K$ satisfies \eqref{regularity} too. By \eqref{Holder}, we have
 \begin{align*}
 \big\|K(x)\big\|_{X}^2
 &=\int_\Gamma\sup_{\f\in\H_{\e, \d}}\big|\f_t(x+y)\big|^2\,\frac{dy\,dt}{t^{d+1}}\\
 &\le\int_\Gamma\big[\frac1{t^d}\,\frac1{\big(1+\frac{|x+y|}t\big)^{d+\e}}\big]^2\,\frac{dy\,dt}{t^{d+1}}\\
&\les_{d,\e}\int_0^\8\big[\frac1{t^d}\,\frac1{\big(1+\frac{|x|}t\big)^{d+\e}}\big]^2\,\frac{dt}{t}
\les_{d, \e} \frac{1}{|x|^{2d}}.
 \end{align*}
This gives the first estimate of \eqref{regularity}. For the second, let $x, z\in\real^d$ with  $|x|>2|z|$. Applying \eqref{Holder} once more, we get
 \begin{align*}
 \big\|K(x+z)-K(x)\big\|_{X}^2
 \les_{\e,\d}\int_\Gamma\big[\frac1{t^{d+\d}}\,\frac{|z|^\d}{\big(1+\frac{|x+y|}t\big)^{d+\e+\d}}\big]^2\,\frac{dy\,dt}{t^{d+1}}
\les_{d, \e,\d} \frac{|z|^{2\d}}{|x|^{2(d+\d)}}.
 \end{align*}
By \cite{Wilson07},  $S_{\e, \d}$ is bounded on $L_2(\real^d)$. Thus we deduce the following

\begin{lem}\label{CZb}
 Let  $1\le p\le2$. Then
 $$\|S_{\e, \d}(f)\|_p\les_{d,\e,\d} \|f\|_{H_p}, \quad f\in H_p(\real^d).$$
\end{lem}

Our proof of Theorem~\ref{fML} (ii) for $p< 2$ is modelled on Mei's argument \cite{Mei2007} (see also the proof of Theorem~1.3 of \cite{XXX}). We will need a variant of the usual BMO space. For any locally integrable function $f$ on $\real^d$ define
 $$f^\sharp(x)=\sup_{x\in Q}\Big(\frac{1}{|Q|}\int_Q\big|f(x)-\la f\ra_Q\big|^2dx\Big)^{\frac 12},$$
and for $2<q\le\8$  let
  $$\mathrm{BMO}_q(\real^d)=\big\{f: f^\sharp\in L_q(\real^d)\big\}\;\text{ and }\; \big\|f\big\|_{\BMO_q}=\| f^\sharp\|_q .$$
Note that $\mathrm{BMO}_\8(\real^d)$ coincides with the usual $\mathrm{BMO}(\real^d)$.

The BMO space is closely related to Carleson measures via the following maximal function
 $$C^\f(f)(x)=\sup_{x\in Q}\Big(\frac{1}{|Q|}\int_{T(Q)}\big|\f_t*f(y)\big|^2\,\frac{dy\,dt}t\Big)^{\frac 12},$$
where $T(Q)=Q\times (0,\,\el(Q)]\subset\real^{d+1}_+$.

\medskip

The following inequality is known, it can be shown by adapting  the proof of \cite[Theorem~IV.4.3]{Stein1993}.

\begin{lem}\label{carleson-BMO}
 Let $\f\in\H_{\e, \d}$ and $f$ be any nice function on $\real^d$. Then
   $C^\f(f)\les_{d,\e,\d} f^\sharp.$
  \end{lem}

We now arrive at the key step of our argument.

\begin{lem}\label{duality1}
Let $\f,\p\in\H_{\e,\d}$ satisfy \eqref{reproduce}. Let $1\le p<2$. Then
 $$
 \Big|\int_{\real^d}fg\Big|\les_{d,\e,\d} \big\|G^\f(f)\big\|_p^{\frac{p}2}\big\|f\big\|_{H_p}^{1-\frac{p}2}\big\|g\big\|_{\BMO_{p'}}
 $$
 for any sufficiently nice functions $f\in H_p(\real^d)$ and $g\in \BMO_{p'}(\real^d)$.

 \end{lem}

\begin{proof}
Fix (sufficiently nice) functions $f\in H_p(\real^d)$ and $g\in \BMO_{p'}(\real^d)$. We need to consider a truncated version of $G^\f(f)$:
 \beq\label{truncated G}
 G(x, t) =\Big(\int_t^\8 |\f_s*f (x)|^2\frac{ds}{s}\Big)^{\frac12}\,,\quad x\in\real^d,\; t\ge0.
 \eeq
By approximation, we can assume that $G(x,t)$ never vanishes. By \eqref{reproduce}, we have
  \begin{align*}
  \int_{\real^d}fg
 &=\int_{\real^{d+1}_+}\f_t*f(x)\, \p_t*g(x)\,\frac{dx\, dt}{t}\\
 &=\int_{\real^{d+1}_+} \big[\f_t*f(x)G(x,t)^{\frac{p-2}2}\big]\cdot
 \big[G(x, t)^{\frac{2-p}2} \p_t*g(x)\big]\,\frac{dx\,dt}{t}.
 \end{align*}
Thus by the Cauchy-Schwarz inequality,
 \begin{align*}
  \Big|\int_{\real^d}fg\Big|\le {\rm A}\cdot {\rm B},
    \end{align*}
  where
   \begin{align*}
 {\rm A}^2&=\int_{\real^{d+1}} G(x,t)^{p-2}|\f_t*f (x)|^2 \,\frac{ dx\,dt}{t},\\
  {\rm B}^2&=\int_{\real^{d+1}} G(x,t)^{2-p}|\p_t*g (x)|^2  \,\frac{ dx\,dt}{t}.
  \end{align*}
The term ${\rm A}$ is estimated as follows
  \begin{align*}
    {\rm A}^2
  =- \int_{\real^{d}}\int_0^\8G(x,t)^{p-2}  \frac{\partial}{\partial t}\big(G(x,t)^2\big)dt\, dx
 =-2 \int_{\real^{d}}\int_0^\8G(x, t)^{p-1} \frac{\partial}{\partial t} G(x, t)\,dt\, dx.
  \end{align*}
Since $G(\cdot,t)$ is decreasing in $t$, $G(\cdot,t)^{p-1}\le G(x, 0)^{p-1}=G^\f(f)(x)^{p-1}$. Thus
  $$
  {\rm A}^2
 \le-2 \int_{\real^{d}}G^\f(f)(x)^{p-1}\int_0^\8  \frac{\partial}{\partial t} G(x,t)\,dt\, dx
  =2\int_{\real^{d}}G^\f(f)(x) ^pds=2\|G^\f(f)\|_{p}^p\,.
 $$

The estimate of B is harder. We will need two more variants of $S_{\e, \d}(f)$. The first one is defined as before for $G(\cdot,t)$:
 $$
 S(x, t)^2
  =\int_t^\8\int_{|y-x|<s-\frac{t}2}\,\sup_{\f\in\H_{\e, \d}}\big|\f_s*f(y)\big|^2\,\frac{dy\,ds}{s^{d+1}}\,,\quad x\in\real^d,\; t\ge0.
 $$
To introduce the second, let $\D_k$ be the family of dyadic cubes of side length $2^{-k}$, and let $c_Q$ denote the center of a cube $Q$. Define
$$
 \mathbb S(x, k)^2=\int_{\sqrt d\,2^{-k}}^\8\int_{|y-c_Q|<s}\,\sup_{\f\in\H_{\e, \d}}\big|\f_s*f(y)\big|^2\,\frac{dy\,ds}{s^{d+1}}\,,\;\;\textrm{ if }\;\;
 x\in Q\in\D_k,\; k\in\ent.
$$
By definition, we have
 \begin{enumerate}[$\bullet$]
 \item $\mathbb S(\cdot, k)$ is increasing in $k$;
 \item  $\mathbb S(\cdot, k)$ is constant on every $Q\in\D_k$;
 \item $\mathbb S(\cdot, -\8)=0$  and $\mathbb S(\cdot, \8)=S(x, 0)=S_{\e, \d}(f)$.
 \end{enumerate}
 On the other hand, if $s\ge t\ge \sqrt d\,2^{-k}$ and $x\in Q\in\D_k$, then $B(x, s-\frac{t}2)\subset B(c_Q, s)$. Here $B(x, r)$ denotes the ball of center $x$ and radius $r$. It then follows that
 $$S(\cdot, t)\le \mathbb S(\cdot, k)\;\text{ on every}\; Q\in\D_k\;\text{ whenever } \; t\ge \sqrt d\,2^{-k}.$$
Here, the crucial observation is the elementary pointwise inequality:  $G(x, t)\les_{d,\e, \d} S(x, t).$ This inequality is easily proved  by the arguments of \cite{Wilson07}. Indeed, by a lemma due to Uchiyama \cite{Uch} (see also Lemma~3 of \cite{Wilson07}), we can assume that the function $\f$ defining $G(x, t)$ in \eqref{truncated G} is supported in the unit ball of $\real^d$. Then we get $G(x, t)\les_{d,\e, \d} S(x, t)$ exactly as Wilson did on page 784 of \cite{Wilson07}.

After these preparations, we are ready to estimate the term B.  We have
   \begin{align*}
  {\rm B}^2
  &\les_{d,\e,\d}\int_{\real^{d+1}} S(x,t)^{2-p}|\p_t*g (x)|^2  \,\frac{ dx\,dt}{t}\\
  &=\sum_{k\in\ent}\sum_{Q\in\D_k}\int_{Q}\int_{\sqrt d\,2^{-k}}^{\sqrt d\,2^{-k+1}}S(x,t)^{2-p}|\p_t*g (x)|^2 \,\frac{dt}{t}\,dx \\
 &\le \sum_{k\in\ent}\sum_{Q\in\D_k}\int_{Q}\int_{\sqrt d\,2^{-k}}^{\sqrt d\,2^{-k+1}}\mathbb S(x, k)^{2-p}
 |\p_t*g (x)|^2 \,\frac{dt}{t}\,dx\\
 &=\int_{\real^d}\sum_{k\in\ent}\sum_{j\le k} D(x, j) \int_{\sqrt d\,2^{-k}}^{\sqrt d\,2^{-k+1}}
 |\p_t*g (x)|^2 \,\frac{dt}{t} \,dx,
  \end{align*}
where $D(x, k)=\mathbb S(x, k)^{2-p} -\mathbb S(x, k-1)^{2-p}$. Thus
   \begin{align*}
 {\rm B}^2
 &\les_{d,\e,\d}\int_{\real^d}\sum_{j} D(x, j)\sum_{k\ge j}   \int_{\sqrt d\,2^{-k}}^{\sqrt d\,2^{-k+1}}
 |\p_t*g (x)|^2 \,\frac{dt}{t} \,dx \\
 &=\sum_{j} \sum_{Q\in\D_j}\int_QD(x, j)\int_{0}^{\sqrt d\,2^{-j+1}}
 |\p_t*g (x)|^2 \,\frac{dt}{t} \,dx.
   \end{align*}
 Since $D(\cdot, j)$ is constant on every $Q\in\D_j$, we have
  \begin{align*}
 {\rm B}^2
 &\les_{d,\e,\d}\sum_{j} \sum_{Q\in\D_j}D(x, j)\un_Q(x)\int_Q\int_{0}^{2\sqrt{d}\,\el(Q)}
 |\p_t*g (x)|^2 \,\frac{dt}{t} \,dx \\
 &\les_{d}\sum_{j} \sum_{Q\in\D_j}D(x, j)\un_Q(x)\min_{y\in Q}C^\p(g)(y)^2\,|Q|\\
 &\le\sum_{j} \sum_{Q\in\D_j}\int_QD(x, j)C^\p(g)(x)^2dx\\
 &\le\int_{\real^d}\sum_{j}D(x, j)C^\p(g)(x)^2dx\\
 &=\int_{\real^d} S_{\e, \d}(f)(x)^{2-p}C^\p(g)(x)^2dx\\
 &\le\big\|S_{\e, \d}(f)\big\|_p^{2-p}\big\|C^\p(g)\big\|_{p'}^{2}.
  \end{align*}
 By Lemma~\ref{CZb},
  $\big\|S_{\e, \d}(f)\big\|_p\les_{d,\e,\d}\big\|f\big\|_{H_p}.$
 Hence,
  $$ {\rm B} \les_{d,\e,\d}\big\|f\big\|_{H_p}^{1-\frac{p}2}\big\|C^\p(g)\big\|_{p'}.$$
 Combining the estimates of A and B together with Lemma~\ref{carleson-BMO}, we get the desired assertion.
  \end{proof}

 \medskip

The preceding lemma implies the following
  $$
 \Big|\int_{\real^d}fg\Big|\les_{d}\big\|f\big\|_{H_p} \big\|g\big\|_{\BMO_{p'}}.
 $$
This shows that every function $g\in\BMO_{p'}(\real^d)$ induces a continuous linear functional on $H_p(\real^d)$. Like the $H_1$-BMO duality theorem, the converse is true too. The following lemma is known, its noncommutative analogue is  \cite[Theorem~4.4]{Mei2007}. We include a proof by adapting Mei's argument.

\begin{lem}\label{duality2}
 Let $1\le p<2$. Then every continuous functional $\el$ on $H_p(\real^d)$ is represented by a function $g\in\BMO_{p'}(\real^d)$:
  $$\el(f)=\int_{\real^d}fg,\quad \forall\, f\in H_1(\real^d)\cap L_2(\real^d).$$
 Moreover,
 $$(2-p)\big\|g\big\|_{\BMO_{p'}}\les_d\big\|\el\big\|_{H_p(\real^d)^*}\les_d \big\|g\big\|_{\BMO_{p'}}.$$
\end{lem}

\begin{proof}
We will use the characterization of $H_p(\real^d)$  by the $g$-function defined by $\eqref{gradient G}$, namely
 $$\big\|f\big\|_{H_p}\approx_d \big\|G_\nabla^{\mathbb{P}}(f)\big\|_p.$$
Let $\el\in H_p(\real^d)^*$. Then by the Hahn-Banach theorem, there exist $d+1$ functions $h_i$ on the upper half space $\real^{d+1}_+$ such that
 $$\Big(\int_{\real^d}\big(\sum_{i=1}^{d+1}\int_0^\8|h_{i}(y, t)|^2\frac{dt}t\big)^{\frac{p'}2}\,dy\Big)^{\frac1{p'}}\approx_d\big\|\el\big\|_{H_p(\real^d)^*}$$
 and (with $x_{d+1}=t$)
  \begin{align*}
  \el(f)
  =\int_{\real^d}\int_0^\8\Big(\sum_{i=1}^{d+1}t\frac{\partial}{\partial x_i}\mathbb{P}_t(f)(y) h_{i}(y, t)\Big)\frac{dt}t\,dy
  =\int_{\real^d}f(x)g(x)dx,
    \end{align*}
  where
   $$g(x)=\int_{\real^d}\int_0^\8\Big(\sum_{i=1}^{d+1}t\frac{\partial}{\partial x_i}\mathbb{P}_t(y-x)(y) h_{i}(y, t)\Big)\frac{dt}t\,dy.$$
It remains to show that $g\in\BMO_{p'}(\real^d)$. All the $d+1$ terms on the right hand side are treated in the same way, so we need only to deal with one of them, say the $i$-th term. For notational simplicity, let
 $$\f_t(x)=t\frac{\partial}{\partial x_i}\mathbb{P}_t(-x),\quad h=h_i$$
and (with some abuse of notation)
 $$g=\int_0^\8\f_t*h(\cdot, t)\,\frac{dt}t,\quad H(y)=\Big(\int_0^\8| h(y, t)|^2\,\frac{dt}t\Big)^{\frac12}.$$
Given $x\in\real^d$ and  a cube $Q$ containing $x$, let
 $$a_Q=\frac1{|Q|}\int_Q\int_{\real^d}\int_0^\8\f_t(y-z)h(y, t)\un_{(2Q)^c}(y)\,\frac{dt}t\,dy\,dz.$$
Then
 \begin{align*}
 g(u)-a_Q
 &=\int_{\real^d}\int_0^\8\f_t(y-u)h(y, t)\un_{2Q}(y)\,\frac{dt}t\,dy\\
 &\;\;+\int_{\real^d}\int_0^\8\big[ \frac1{|Q|}\int_Q\big(\f_t(y-u)-\f_t(y-z)\big)dz \big]h(y, t)\un_{(2Q)^c}(y)\,\frac{dt}t\,dy\,dz\\
 &{\mathop=^{\rm def}}\; A(u)+B(u).
  \end{align*}
By the Plancherel identity and the Cauchy-Schwarz inequality, letting $\tilde h_t(y)=h(y, t)\un_{2Q}(y)$, we have
  \begin{align*}
  \int_Q|A(u)|^2du
  &\le  \int_{\real^d}|A(u)|^2du= \int_{\real^d}\Big|\int_0^\8\wh\f(t\xi)\,\wh {\tilde h}_t(\xi)\,\frac{dt}t\Big|^2d\xi\\
  &\le \int_{\real^d}\int_0^\8|\wh\f(t\xi)|^2\,\frac{dt}t\,\int_0^\8|\wh {\tilde h}_t(\xi)|^2\,\frac{dt}t\,d\xi\\
  &\les_d \int_{2Q}\int_0^\8| h(y, t)|^2\,\frac{dt}t\,dy.
    \end{align*}
  It then follows that
   $$\sup_{x\in Q}\frac1{|Q|}\int_Q|A(u)|^2du\les_d M(H^2)(x).$$
  We turn to the term $B$. Let $c$ be the center of $Q$ and $u\in Q$, then
  \begin{align*}
   |B(u)|
  &\les_d \int_{(2Q)^c}\int_0^\8\frac{\el(Q)}{(|y-c|+t)^{d+1}}\,|h(y, t)|\,\frac{dt}t\,dy\\
  &\les_d \int_{(2Q)^c}\frac{\el(Q)}{|y-c|^{d+1}}\Big(\int_0^\8|h(y, t)|^2\frac{dt}t\Big)^{\frac12}dy\\
  &\les_d \sum_{k=1}^\8 2^{-k}\, \frac1{\big(2^{k}\el(Q)\big)^d}\int_{2^{k-1}\el(Q)\le |y-c|<2^{k}\el(Q)} H(y)dy\\
  &\les_d M(H)(x).
   \end{align*}
 Thus
   $$\sup_{x\in Q}\Big(\frac1{|Q|}\int_Q|B(u)|^2du\Big)^{\frac12}\les_d M(H)(x).$$
Combining the preceding estimates, we get
  \begin{align*}
  g^\sharp(x)
  &\les\sup_{x\in Q}\Big(\frac1{|Q|}\int_Q|g(u)-a_Q|^2du\Big)^{\frac12}\\
  &\les\sup_{x\in Q}\Big(\frac1{|Q|}\int_Q|A(u)|^2du\Big)^{\frac12}+\sup_{x\in Q}\Big(\frac1{|Q|}\int_Q|B(u)|^2du\Big)^{\frac12}\\
   &\les_d \big(M(H^2)(x)\big)^{\frac12}+ M(H)(x)\les_d  \big(M(H^2)(x)\big)^{\frac12}.
    \end{align*}
  Hence,
   $$\big\|g^\sharp\big\|_{p'}\les_d\big\|\big(M(H^2)\big)^{\frac12}\big\|_{p'}\les_d\frac1{2-p}\, \big\|H\big\|_{p'}\les_d\frac1{2-p}\, \big\|\el\big\|_{H_p(\real^d)^*}.$$
This is the desired inequality.
\end{proof}

It is now easy to show part (ii) of Theorem~\ref{fML} for $1\le p\le2$.

 \begin{proof}[Proof of Theorem~\ref{fML} (ii) for $p<2$]
  Using Lemma~\ref{duality2} and taking the supremum in the inequality of Lemma~\ref{duality1} over all $g$ with $\big\|g\big\|_{\BMO_{p'}}\le1$, we obtain
  $$\big\|f\big\|_{H_p}\les_{d,\e,\d}\frac1{2-p}\,  \big\|G^\f(f)\big\|_p^{\frac{p}2}\big\|f\big\|_{H_p}^{1-\frac{p}2},$$
 whence
 $$\big\|f\big\|_{H_p}\les_{d,\e,\d}(2-p)^{-\frac2p}\,  \big\|G^\f(f)\big\|_p.$$
Since $\big\|f\big\|_{p}\le \big\|f\big\|_{H_p}$, we deduce
 $$\big\|f\big\|_{p}\les_{d,\e,\d}(2-p)^{-\frac2p}\,  \big\|G^\f(f)\big\|_p.$$
 This implies  $\mathsf{L}^{\f}_{t,p}\les_{d, \e, \d}1$ for $p\le\frac32$. For $\frac32<p<2$ we use duality and $\mathsf{L}^{\p}_{c,p'}\les_{d, \e, \d}\sqrt{p'}$ to conclude that $\mathsf{L}^{\f}_{t,p}\les_{d, \e, \d}\sqrt{p'}\les_{d, \e, \d}1$ too.
 \end{proof}

 \begin{proof}[Proof of Corollary~\ref{H1}]
 In the previous proof we have obtained $\big\|f\big\|_{H_p}\les_{d,\e,\d} \big\|G^\f(f)\big\|_p$ for $1\le p\le2$. The converse inequality is contained in Lemma~\ref{CZ}. On the other hand, by Lemma~\ref{CZb}, we get
  $$\big\|S^\p(f)\big\|_p\les_{d,\e,\d}\big\|f\big\|_{H_p}\les_{d,\e,\d} \big\|G^\f(f)\big\|_p.$$
 Hence the corollary is proved.
  \end{proof}


\section{Proof of Theorem~\ref{PH ML}}\label{Proof of Theorem PH ML}


We have seen in the previous section that the $g$-function  $G^{ \mathbb{P}}$ can be expressed as a singular integral operator. Equivalently, $G^{ \mathbb{P}}$ can be also written as a Fourier multiplier with values in $L_2(\real_+)$ (recalling that $\real_+$ is equipped with the measure $\frac{dt}t$). Let $\f:\real^d\to  L_2(\real_+)$ be the function defined by $\f(\xi)(t)=-2\pi t|\xi|e^{-2\pi t|\xi|}$ for $\xi\in\real^d$ and $t>0$. Let  $T_\f$  be the Fourier multiplier introduced in section~\ref{A variant of de Leeuw's multiplier theorem} (with $X=\com$ and $Y=L_2(\real_+)$). Then
 $$G^{ \mathbb{P}}(f)(x)=||T_\f(f)(x)\|_{L_2(\real_+)},\quad x\in\real^d,\; f\in L_p(\real^d).$$
It is clear that this symbol $\f$ satisfies the assumption of Theorem~\ref{de Leeuw} (see also Remark~\ref{de Leeuw bis}). Thus the results in that section apply to $\f$.

The corresponding periodic Fourier multiplier $M_\f$ gives rise to the $g$-function defined by the circular Poisson semigroup on $\T^d$:
 $$\mathsf{P}_r(f)=\sum_{m\in\ent^d}\wh f(m)r^{|m|}z^m,\quad f\in L_p(\T^d).$$
The associated $g$-function is defined by
 \beq\label{gT}
 G^{\mathsf{P}}(f)=\Big(\int_0^1(1-r)\big|\frac{d}{dr} \mathsf{P}_r(f)\big|^2\,dr\Big)^{\frac12}\,.
 \eeq
By the change of variables $r=e^{-2\pi t}$, elementary computations show that for any $1\le p\le\8$
 $$\|G^{\mathsf{P}}(f)\|_{L_p(\T^d)}\approx_d \|M_\f(f)\|_{L_p(\T^d;L_2(\real_+))},\quad f\in L_p(\T^d).$$
We refer to \cite[Section~8]{CXY2012} for more details.

\medskip

Now we proceed to the proof of Theorem~\ref{PH ML}. For clarity, we will divide this proof into two subsections. Let us first make an elementary observation as a prelude.  $\{\mathbb{P}_t\}_{t>0}$ is the Poisson semigroup subordinated to $\{\mathbb{H}_t\}_{t>0}$ in Bochner's sense:
 $$
 \mathbb{P}_t(f)=\frac{1}{\sqrt\pi}\,\int_0^\infty \frac{e^{-s}}{\sqrt s}\,\mathbb{H}_{\frac{t^2}{4s}}(f)ds.
$$
This formula  immediately implies
 $$\mathsf{L}^{\mathbb{H}}_{c, p}\ges \mathsf{L}^{\mathbb{P}}_{c, p}\quad \text{ and }\quad \mathsf{L}^{\mathbb{H}}_{t, p}\les \mathsf{L}^{\mathbb{P}}_{t, p}.$$
Combining this with the historical comments at the end of subsection~\ref{Littlewood-Paley g-function inequality},  it remains for us to show that $\mathsf{L}^{\mathbb{P}}_{c, p}\ges \max(\sqrt p,\,p')$, and $\mathsf{L}^{\mathbb{H}}_{t, p}\ges \sqrt p$ for $p>2$. The former is already contained in  \cite{HFC-LPS}, but we will reproduce the proof there for the convenience of the reader.


\subsection{Proof of $\mathsf{L}^{\mathbb{P}}_{c, p}\ges \max(\sqrt p,\,p')$ }


By the discussion at the beginning of this section and Corollary~\ref{subgroup}, the constant $\mathsf{L}^{\mathbb{P}}_{c, p}$ increases in the dimension $d$. So it suffices to consider the case $d=1$. This inequality for $p\le 2$ is well known. It can be easily proved as follows. Fix $s>0$ and let $f=\mathbb{P}_s$. Then
 $$t\frac{\partial}{\partial t} \mathbb{P}_t(f)(x)=\frac{t}\pi\, \frac{x^2-(t+s)^2}{(x^2+(t+s)^2)^2},\quad x\in\real.$$
For $x\ge 6s$, we have
 \begin{align*}
 G^{ \mathbb{P}}(f)(x)\ge \Big(\int_{\frac{x}3-s}^{\frac{x}2-s}\big|t\frac{\partial}{\partial t} \mathbb{P}_t(f)(x)\big|^2\,\frac{dt}t\Big)^{\frac12}\ges\frac1x.
\end{align*}
Thus
 $$\big\|G^{ \mathbb{P}}(f)\big\|_p\ges  \Big(\int_{6s}^{\8}\frac1{x^p}\,dx\Big)^{\frac1p}\ges \frac{s^{-\frac1{p'}}}{p-1}.$$
On the other hand,
 $$\big\| f\big\|_p\approx s^{-\frac1{p'}}.$$
Hence, $\mathsf{L}^{\mathbb{P}}_{c,p}\ges p'$.

\medskip

Unfortunately, the above simple argument does not apply to the case $p>2$. Our proof for the latter is much harder.  By the discussion at the beginning of this section,  it is equivalent to considering the torus $\T$ and the $g$-function defined by \eqref{gT}.
Recall that
 $$\mathsf{P}_r(\t)=\frac{1-r^2}{1-2r\cos\t+r^2}\,.$$
  It is shown in \cite{LP0} that the inequality
 $$
 \|G^{\mathsf{P}}(f)\|_{L_p(\T)}\le \mathsf{L}^{\mathsf{P}}_{c,p}\|f\|_{L_p(\T)}
 $$
is equivalent to the corresponding dyadic martingale inequality on $\O=\{-1, 1\}^\nat$. It is well known that the relevant constant in the latter martingale inequality is of order $\sqrt p$ as $p\to\8$. To reduce the determination of  optimal order of $\mathsf{L}^{\mathsf{P}}_{c,p}$ to  the martingale case, we need to refine an argument in the proof of \cite[Theorem~3.1]{LP0} whose idea originated from \cite{Bourgain83}.

Keeping the notation there, let $M=(M_k)_{0\le k\le K}$ be a finite dyadic martingale and
 $$M_k-M_{k-1}=d_k(\e_1,\cdots, \e_{k-1})\,\e_k,$$
where $(\e_k)$ are the coordinate functions of $\O$.
The transformation $\e_k={\rm sgn}(\cos\t_k)$ establishes a measure preserving embedding of $\O$ into $\T^\nat$. Accordingly, define
 \begin{align*}
  a_k(e^{\i\t_1},\cdots, e^{\i\t_{k-1}})
  &=d_k({\rm sgn}(\cos\t_1),\cdots, {\rm sgn}(\cos\t_{k-1})),\\
  b_k(e^{\i\t_{k}})&={\rm sgn}(\cos\t_k).
  \end{align*}
Note that to enlighten notation, we write an element $z\in\T$ as $z=e^{-\i \t}$, so identify $\T$ with $[-\pi,\, \pi)$, a slightly different convention from the one of section~\ref{A variant of de Leeuw's multiplier theorem}.

Given $(n_k)$ a rapidly increasing sequence of positive integers, put
 \begin{align*}
  a_{k, (n)}(e^{\i\t})
  &=a_{k, (n)}(e^{\i\t}; e^{\i\t_1},\cdots, e^{\i\t_{k-1}})
  =a_k(e^{\i(\t_1+n_1\t)},\cdots, e^{\i(\t_{k-1}+n_{k-1}\t)}),\\
  b_{k, (n)}(e^{\i\t})
  &=b_{k, (n)}(e^{\i\t};e^{\i\t_{k}})
  =b_k(e^{\i(\t_{k}+n_k\t)}),\\
  f_{(n)}(e^{\i\t})
  &=f_{(n)}(e^{\i\t}; e^{\i\t_1},\cdots, e^{\i\t_{K}})
  =\sum_{k=1}^Ka_{k, (n)}(e^{\i\t})b_{k, (n)}(e^{\i\t}).
  \end{align*}
 The functions $f_{(n)}$, $a_{k, (n)}$ and $b_{k, (n)}$ are viewed as functions on $\T$ for each $(\t_1, \cdots, \t_K)$ arbitrarily fixed. Furthermore, by approximation, we can assume that all $a_k$ and $b_k$ are polynomials. Then, if the sequence $(n_k)$ rapidly increases, Lemmas~3.4 and 3.5 of \cite{LP0} imply
  $$\frac12\, G^{\mathsf{P}}(f_{(n)})\le \Big(\sum_{k=1}^K|a_{k, (n)}|^2\,G^{\mathsf{P}}(b_{k, (n)})^2\Big)^{\frac12}\le 2G^{\mathsf{P}}(f_{(n)}).$$
 Therefore,
  \beq\label{inter}
  \Big\|\Big(\sum_{k=1}^K|a_{k, (n)}|^2\,G^{\mathsf{P}}(b_{k, (n)})^2\Big)^{\frac12}\Big\|_{L_p(\T)}
  \le 2 \mathsf{L}^{\mathsf{P}}_{c, p}\big\|f\big\|_{L_p(\T)}\,.
  \eeq
The discussion so far comes from \cite{LP0}. Now we require a finer analysis of the $g$-function $G^{\mathsf{P}}(b_{k, (n)})$.  To this end we write the Fourier series of the function $b={\rm sgn}(\cos\t)$:
  $$b(e^{\i\t})=\frac2\pi\,\sum_{j=0}^\8\frac{(-1)^j}{2j+1}\,\big[e^{\i(2j+1)\t}+e^{-\i(2j+1)\t}\big].$$
Then
 $$\frac{d}{dr}\mathsf{P}_r(b_{k, (n)})(e^{\i\t})=\frac4\pi\,n_kr^{n_k-1}{\rm Re}\Big(\sum_{j=0}^\8(-1)^j r^{2n_kj}e^{\i(2j+1)(\t_k+n_k\t)}\Big).$$
The real part on the right side is easy to compute. Indeed, letting $\rho=r^{2n_k}$ and $\eta=\t_k+n_k\t$, we have
 $$
  {\rm Re}\Big(e^{\i\eta}\,\sum_{j=0}^\8(-1)^j \rho^{j}e^{\i 2j\eta}\Big)
  ={\rm Re}\Big(\frac{e^{\i\eta}}{1+\rho e^{\i2\eta}}\Big)
  =\frac{1+\rho}{1+2\rho\cos\eta +\rho^2}\,\cos\eta.
  $$
As
 $$\frac{1+\rho}{1+2\rho\cos\eta +\rho^2}\ge \frac{1+\rho}{(1+\rho)^2}\ge\frac12,$$
it then follows that
 $$\Big| {\rm Re}\Big(e^{\i\eta}\,\sum_{j=0}^\8(-1)^j \rho^{j}e^{\i 2j\eta}\Big)\Big|\ge \frac{|\cos\eta|}2.$$
 Therefore, we deduce
 $$\left|\frac{d}{dr}\mathsf{P}_r(b_{k, (n)})(e^{\i\t})\right|^2\ges n^2_kr^{2(n_k-1)}\cos^2(\t_k+n_k\t).$$
Thus
 \begin{align*}
 G^{\mathsf{P}}(b_{k, (n)})^2
 &\ges\cos^2(\t_k+n_k\t)\,n^2_k\int_0^1(1-r)^{2-1}r^{2(n_k-1)}dr \\
 &\approx \big[1+{\rm O}(\frac1{n_k})\big]\cos^2(\t_k+n_k\t).
 \end{align*}
Now lifting both sides of \eqref{inter} to power $p$, then integrating the resulting inequality over $\T^K$ with respect to $(\t_1,\cdots, \t_K)$, we get
 \begin{align*}
 \int_\T&\int_{\T^K}\Big(\sum_{k=1}^K|a_{k,(n)}(e^{\i(\t_1+n_1\t)},\cdots, e^{\i(\t_{k-1}+n_{k-1}\t)})|^2\big[1+{\rm O}(\frac1{n_k})\big]\cos^2(\t_k+n_k\t)\Big)^{\frac{p}2}d\t_1\cdots d\t_Kd\t\\
  &\le \big(C \mathsf{L}^{\mathsf{P}}_{c, p}\big)^p\int_\T\int_{\T^K}\big|f_{(n)}(e^{\i(\t_1+n_1\t)}, \cdots, e^{\i(\t_K+n_K\t)})\big|^pd\t_1\cdots d\t_Kd\t\,.
 \end{align*}
 For each fixed $\t$, the change of variables $(\t_1,\cdots, \t_K)\mapsto(\t_1-n_1\t,\cdots, \t_K-n_K\t)$ being a measure preserving transformation of $\T^K$, we deduce
  \begin{align*}
\int_{\T^K}&\Big(\sum_{k=1}^K|a_{k,(n)}(e^{\i\t_1},\cdots, e^{\i\t_{k-1}})|^2\,\big[1+{\rm O}(\frac1{n_k})\big]\cos^2\t_k\Big)^{\frac{p}2}\,d\t_1\cdots d\t_K\\
  &\le \big(C\mathsf{L}^{\mathsf{P}}_{c, p}\big)^p\int_{\T^K}\big|f_{(n)}(e^{\i\t_1}, \cdots, e^{\i\t_K})\big|^p\,d\t_1\cdots d\t_K\,.
 \end{align*}
Letting $n_1\to\8$, we get
  \begin{align*}
\int_{\T^K}&\Big(\sum_{k=1}^K|d_{k}({\rm sgn}(\cos\t_1),\cdots, {\rm sgn}(\cos\t_{k-1}))|^2\,\cos^2\t_k\Big)^{\frac{p}2}\,d\t_1\cdots d\t_K
  \le \big(C\, \mathsf{L}^{\mathsf{P}}_{c,p}\big)^p\big\|M_K\big\|^p_{L_p(\O)}\,.
 \end{align*}

Now we consider an elementary example where $M$ is simple random walk stopped at $\pm 2$, namely
 $$d_k=\un_{\{\tau\ge k\}}\;\text{ with }\;
 \tau=\inf\big\{k: \big|\sum_{j=1}^k\e_j\big|=2\big\}.$$
 Note that the probability of the event $\{\tau=j\}$ is zero for odd $j$ and $2^{-\frac{j}2}$ for even $j$. On the other hand, recalling $\e_k={\rm sgn}(\cos\t_k)$ and letting
  $$A_j=\big\{\tau=j,\; |\cos\t_k|\ge \frac1{\sqrt 2}\,,\; 1\le k\le j\big\},$$
we easily check that the probability of $A_j$ is $8^{-\frac{j}2}$ for even $j$. Thus
 \begin{align*}
 \sum_{k=1}^K|d_{k}(\e_1,\cdots, \e_{k-1}))|^2\,\cos^2\t_k
 \ge \un_{A_j}\sum_{k=1}^j\un_{\{\tau\ge k\}}\cos^2\t_k
 \ge \frac{j}2\,\un_{A_j},
 \end{align*}
consequently, for $K=2J$
 \begin{align*}
 \int_{\T^K}&\Big(\sum_{k=1}^K|d_{k}({\rm sgn}(\cos\t_1),\cdots, {\rm sgn}(\cos\t_{k-1}))|^2\,\cos^2\t_k\Big)^{\frac{p}2}\,d\t_1\cdots d\t_K
 \ges\sum_{m=1}^J j^{\frac{p}2} 8^{-j}\ge c^p p^{\frac{p}2}.
  \end{align*}
Noting that $|M_K|\le2$ and combining all the previous inequalities together, we finally obtain
 $$\mathsf{L}^{\mathsf{P}}_{c, p}\ges \sqrt p.$$


\subsection{Proof of $\mathsf{L}^{\mathbb{H}}_{t, p}\ges \sqrt p$  for $p>2$}


Again, it suffices to consider the torus case. The $g$-function relative to the heat semigroup on $\T$ is defined by
 $$G^{\mathsf{H}}(f)=\Big(\int_0^1(1-r)|\frac{d}{dr} \mathsf{H}_r(f)|^2dr\Big)^{\frac12}\,,$$
where
 $$\mathsf{H}_r(f)(\t)=\sum_{n\in\ent}\wh f(n)\,r^{n^2}e^{\i n\t}.$$
We will need the following elementary inequality that is known to experts. Let  $a=(a_k)$ be a finite complex sequence  and $f=\sum_{k} a_k e^{\i 2^k\t}$. Then
\beq\label{lacunary g-function}
 \big\|G^{\mathsf{H}}(f)\big\|_{L_p(\T)}\approx\|a\|_{\el_2}\,.
\eeq
See \cite{CHLM} for related results in a more general setting.
The proof is easy:
 \begin{align*}
 G^{\mathsf{H}}(f)(e^{\i \t})^2
 &=\int_0^1(1-r)\big|\sum_{k} a_k4^k r^{4^k-1}e^{\i 2^k\t}\big|^2dr\\
 &\le\sum_{j,k} a_j\bar a_k \frac{4^{j+k}}{(4^j+4^k-1)^2}\\
 &\le\sum_{j,k} |a_j|^2\frac{4^{j+k}}{(4^j+4^k-1)^2}\\
 &\les \sum_{j} |a_j|^2\,.
 \end{align*}
This implies
 $$\big\|G^{\mathsf{H}}(f)\big\|_{L_p(\T)}\les\|a\|_{\el_2}.$$
However, for $2\le p\le\8$,
 $$\big\|G^{\mathsf{H}}(f)\big\|_{L_p(\T)}\ge\big\|G^{\mathsf{H}}(f)\big\|_{L_2(\T)}\approx\|a\|_{\el_2}.$$
Therefore, for $2\le p\le\8$,
 $$\big\|G^{\mathsf{H}}(f)\big\|_{L_p(\T)}\approx\|a\|_{\el_2}.$$
The remaining case $1\le p<2$ then follows from the H\"older inequality.

\medskip

Now it is easy to show $\mathsf{L}^{\mathsf{H}}_{t, p}\ges \sqrt p$ for $p>2$. Indeed,  \eqref{lacunary g-function}  shows that $\mathsf{L}^{\mathsf{H}}_{t, p}$ dominates the best constant $C_p$ in the following inequality
 $$\Big\|\sum_{k\ge1} a_ke^{\i 2^k\t}\Big\|_{L_p(\T)}\le C_p\|a\|_{\el_2}$$
for any finite sequence $a=(a_k)$. It is well known that the best constant $C_p$ in the above inequality is of order $\sqrt p$ as $p\to\8$ (see \cite{Rudin}); this fact can be also seen from the equivalence, up to universal constants,  between this inequality and the classical Khintchine inequality (cf. \cite{Pisier}). Thus $\mathsf{L}^{\mathbb{H}}_{t, p}\ges \sqrt p$. This completes the proof of Theorem~\ref{PH ML}.


\section{Proofs of Theorem~\ref{RvsT}, Theorem~\ref{multi-parameter} and Corollary~\ref{DLP-const}}\label{Proof of Theorem RvsT}


We begin with the proof of Theorem~\ref{RvsT}.

\begin{proof}[Proof of Theorem~\ref{RvsT}]
 Let $H=\el_2(\Delta)$ be the Hilbert space indexed by the family $\Delta$ of dyadic rectangles and $\{e_R\}_{R\in\Delta}$ be its canonical basis. Let $\f:\real^d\to H$ be the function given by
 $$\f=\sum_{R\in\Delta}\un_{R}\,e_R.$$
We will apply Theorem~\ref{de Leeuw} to the case where $X=\com$ and $Y=H$. With the notation introduced in section~\ref{A variant of de Leeuw's multiplier theorem}, for any $f\in L_p(\real^d)$ we have
 $$S^\Delta(f)=\|T_\f(f)\|_H,$$
so
 $$\mathsf{L}^{\Delta}_{c,p, d}=\|T_\f\|_{p\to p}\quad\text{ and }\quad \mathsf{L}^{\Delta}_{t,p, d}=\|T_\f^{-1}\|_{p\to p}.$$
Note that
 $$\|T_{\f^{(t)}}\|_{p\to p}=\|T_\f\|_{p\to p}\quad\text{ and }\quad \|T^{-1}_{\f^{(t)}}\|_{p\to p}=\|T^{-1}_\f\|_{p\to p},\quad \forall\,t>0.$$
Choose $t$ irrational. Then $\f^{(t)}$ satisfies the assumption of Theorem~\ref{de Leeuw} (i), so
 $$\|M_{\f^{(t)}}\|_{p\to p}\le\|T_{\f^{(t)}}\|_{p\to p}\quad \text{ and }\quad \|M^{-1}_{\f^{(t)}}\|_{p\to p}\le\|T^{-1}_{\f^{(t)}}\|_{p\to p}.$$
Letting $t\to1$, we deduce
 $$\|M_{\f}\|_{p\to p}\le\|T_{\f}\|_{p\to p}\quad \text{ and }\quad \|M^{-1}_{\f}\|_{p\to p}\le\|T^{-1}_{\f}\|_{p\to p}.$$
 This means
 $$\mathsf{L}^{\wt\Delta}_{c,p, d}\le\mathsf{L}^{\Delta}_{c,p, d}\quad\text{ and }\quad \mathsf{L}^{\wt\Delta}_{t,p, d}\le\mathsf{L}^{\Delta}_{t,p, d}.$$
To show the converse inequalities, we note that for any integer $j$ and any $f\in L_p(\T^d)$ we have
 $$\|M_{\f^{(2^j)}}(f)\|_H=S^{\wt\Delta}(f).$$
 This implies
 $$\mathsf{L}^{\wt\Delta}_{c,p, d}=\|M_{\f^{(2^j)}}\|_{p\to p}\quad\text{ and }\quad \mathsf{L}^{\wt\Delta}_{t,p, d}=\|M_{\f^{(2^j)}}^{-1}\|_{p\to p}.$$
 Thus by Theorem~\ref{de Leeuw} (ii), we get
 $$\mathsf{L}^{\Delta}_{c,p, d}\le \liminf_{j\to-\8}\|M_{\f^{(2^j)}}\|_{p\to p}=\mathsf{L}^{\wt\Delta}_{c,p, d}
 \;\text{ and }\; \mathsf{L}^{\Delta}_{t,p, d}\le \liminf_{j\to-\8}\|M^{-1}_{\f^{(2^j)}}\|_{p\to p}=\mathsf{L}^{\wt\Delta}_{t,p, d}.$$
The proof is finished.
 \end{proof}

\begin{rk}
 The above proof is also applicable to the one-sided Littlewood-Paley-Rubio Francia inequality in \cite{rubio, journe}: This inequality and its dual form in \cite{Bourgain85, Osipov11} do not make difference for $\real^d$ and $\T^d$.
\end{rk}

We will need the following lemma for the proof of Theorem~\ref{multi-parameter}. This lemma is of interest for its own right and is the Hilbert space valued extension of a classical theorem of Marcinkiewicz and Zygmund (cf. \cite[Theorem~V.2.7]{GRF}).

Let $(\O, \mu)$ be  a measure space,  $H$ a Hilbert space and $1\le p<\8$. Consider a  linear operator $T$ from $L_p(\O)$ to $L_p(\O; H)$. We are interested in the following inequalities
 $$
 \a^{-1}\|f\|_{L_p(\O)}\le \| T(f)\|_{L_p(\O; H)}\le \b\|f\|_{L_p(\O)},\quad f\in L_p(\O).
 $$
Like for \eqref{Fourier R}, if the first inequality holds for some $\a<\8$, the least $\a$ is denoted by $\|T^{-1}\|_{p\to p}$ while the least $\b$, if exists, is equal to $\|T\|_{p\to p}$. If $\a$ or $\b$ does not exist, $\|T^{-1}\|_{p\to p}$ or $\|T\|_{p\to p}$ is interpreted as infinite.

Given another Hilbert space $K$, consider the tensor ${\rm I}_K\ot T: L_p(\O; K)\to L_p(\O; K\ot H)$, where $ K\ot H$ denotes the tensor Hilbert space. Let $\|{\rm I}_K\ot T\|_{p\to p}$ and  $\|({\rm I}_K\ot T)^{-1}\|_{p\to p}$ be the best constants, if exist, in the following inequalities
  $$
 \a^{-1}\|f\|_{L_p(\O; K)}\le \| {\rm I}_K\ot T(f)\|_{L_p(\O; K\ot H)}\le \b\|f\|_{L_p(\O; K)},\quad f\in L_p(\O; K).
 $$

\begin{lem}\label{MZ}
Let $1\le p<\8$. Then
 $$\|{\rm I}_K\ot T\|_{p\to p}\les \|T\|_{p\to p}\quad\text{ and }\quad \|({\rm I}_K\ot T)^{-1}\|_{p\to p}\les \|T^{-1}\|_{p\to p}.$$
 \end{lem}

\begin{proof}
 It suffices to consider a finite dimensional $K$, say, $K=\el_2^n$. Let $(g_1,\cdots, g_n)$ be a standard complex Gaussian system, the corresponding expectation denoted by $\E$. Then for any $(\a_1, \cdots, \a_n)\in\com^n$ we have
  $$\E\big|\sum_{k=1}^n \a_k g_k\big|^p=\g_p^p \big(\sum_{k=1}^n |\a_k|^2\big)^{\frac{p}2}\;\text { with }\; \g_p^p=\E|g_1|^p.$$
Now let $f=(f_1, \cdots, f_n)\in L_p(\O; \el_2^n)$. Then
 $$
 \|f\|_{L_p(\O; \el_2^n)}^p=\int_\O\big(\sum_{k=1}^n |f_k|^2\big)^{\frac{p}2}=\g_p^{-p}\int_\O\E\big|\sum_{k=1}^n g_kf_k\big|^p.
 $$
On the other hand,
 $$
  \|{\rm I}_K\ot T(f)\|_{L_p(\O; \el_2^n\ot H)}^p=\int_\O\big(\sum_{k=1}^n \|T(f_k)\|_H^2\big)^{\frac{p}2}.
 $$
Recall the following well known fact, the Hilbert-valued Khintchine inequality:
 $$A_p^{-1} \big(\sum_{k=1}^n \|a_k\|_H^2\big)^{\frac12}\le \Big(\E\big\|\sum_{k=1}^n g_ka_k\big\|_H^p\Big)^{\frac1p}\le B_p \big(\sum_{k=1}^n \|a_k\|_H^2\big)^{\frac12},\quad a_k\in H;$$
moreover, $1\les A_p\le1$ for $1\le p\le2$ and $1\le B_p\les\sqrt p\approx \g_p$ for $2\le p<\8$. Using this inequality, we can write (recalling $K=\el_2^n$)
 \begin{align*}
  \|{\rm I}_K\ot T(f)\|_{L_p(\O; \el_2^n\ot H)}^p&\approx\int_\O\E\big\|\sum_{k=1}^n g_kT(f_k)\big\|_H^p,\quad p\le 2,\\
  \|{\rm I}_K\ot T(f)\|_{L_p(\O; \el_2^n\ot H)}^p&\ges\g_p^{-p}\int_\O\E\big\|\sum_{k=1}^n g_kT(f_k)\big\|_H^p,\quad p> 2.
 \end{align*}
Thus for $p\le2$,
  \begin{align*}
  \|{\rm I}_K\ot T(f)\|_{L_p(\O; \el_2^n\ot H)}^p
  &\approx\E\int_\O\big\|\sum_{k=1}^n g_kT(f_k)\big\|_H^p\\
  &\les \|T\|_{p\to p}^p\, \E\int_\O\big|\sum_{k=1}^n g_kf_k\big|^p\\
  &=\|T\|_{p\to p}^p\,\g_p^{p} \|f\|_{L_p(\O; \el_2^n)}^p\\
  &\approx \|T\|_{p\to p}^p\, \|f\|_{L_p(\O; \el_2^n)}^p.
  \end{align*}
 This yields $\|{\rm I}_K\ot T\|_{p\to p}\les \|T\|_{p\to p}$.

 On the other hand, for $p\ge1$, we have
  \begin{align*}
  \|{\rm I}_K\ot T(f)\|_{L_p(\O; \el_2^n\ot H)}^p
  &\ges\g_p^{-p}\,\E\int_\O\big\|\sum_{k=1}^n g_kT(f_k)\big\|_H^p\\
  &\ges \g_p^{-p}\|T^{-1}\|_{p\to p}^{-p}\, \E\int_\O\big|\sum_{k=1}^n g_kf_k\big|^p\\
  &=\g_p^{-p}\|T^{-1}\|_{p\to p}^{-p}\g_p^{p} \|f\|_{L_p(\O; \el_2^n)}^p\\
  &= \|T^{-1}\|_{p\to p}^{-p}\, \|f\|_{L_p(\O; \el_2^n)}^p,
  \end{align*}
 whence $\|({\rm I}_K\ot T)^{-1}\|_{p\to p}\les \|T^{-1}\|_{p\to p}$.

It remains to show $\|{\rm I}_K\ot T\|_{p\to p}\les \|T\|_{p\to p}$ for $p>2$. To this end, we use duality. If $\|T\|_{p\to p}$ is finite, then the adjoint $T^*: L_{p'}(\O; H)\to L_{p'}(\O)$ is a bounded operator with norm equal to $ \|T\|_{p\to p}$. In the same way, we have
 $$\|{\rm I}_K\ot T\|_{p\to p}=\big\|{\rm I}_K\ot T^*:  L_{p'}(\O; K\ot H)\to L_{p'}(\O; K)\big\|.$$
Then arguing as above with $T^*$  and $p'$ instead of $T$ and $p$, respectively, we get
 $$\big\|{\rm I}_K\ot T^*:  L_{p'}(\O; K\ot H)\to L_{p'}(\O; K)\big\|\les \big\|T^*: L_{p'}(\O; H)\to L_{p'}(\O)\big\|,$$
which implies $\|{\rm I}_K\ot T\|_{p\to p}\les \|T\|_{p\to p}$ for $p>2$.
\end{proof}

\begin{rk}\label{MZbis}
Except the last part, the above proof works for $T$ defined on a closed linear subspace $S$ of $L_p(\O)$. Namely, letting $S\ot_pH$ be the closure of $S\ot H$ in $L_p(\O; H)$, then
  \begin{align*}
  \big\|{\rm I}_K\ot T: S\ot_pK\to S\ot_p(K\ot H)\big\|
  &\les \big\|T: S\to S\ot_pH\big\|,\; 1\le p\le2;\\
 \big\|({\rm I}_K\ot T)^{-1}: S\ot_p(K\ot H)\to S\ot_pK\big\|
 &\les \big\||T^{-1}: S\ot_pH\to S\big\|,\; 1\le p<\8.
 \end{align*}
\end{rk}

\begin{proof}[Proof of Theorem~\ref{multi-parameter}]
 By considering $d$-fold tensor products of functions on $\T$, i.e., functions on $\T^d$ of the form $f(z)=f_1(z_1)\cdots f_d(z_d)$, we easily check
  $$\big(\mathsf{L}^{\wt\Delta}_{c,p, 1}\big)^d\le \mathsf{L}^{\wt\Delta}_{c,p, d}\quad\text{ and }\quad
 \big(\mathsf{L}^{\wt\Delta}_{t,p, 1}\big)^d\le \mathsf{L}^{\wt\Delta}_{t,p, d}.$$
Lemma~\ref{MZ} allows us to prove the converse inequalities. We do this for the second, the first being similarly treated; we consider only  the case $d=2$, an iteration argument will then give the general case. Let $\wt\Delta^{(1)}$ be the family of dyadic intervals of $\ent$ and $\wt\Delta^{(2)}$ the same family but with $\ent$ viewed as the second factor of $\ent^2$. Then (with $d=2$)
 $$\wt\Delta=\big\{R^{(1)}\times R^{(2)}\,:\, R^{(i)}\in \wt\Delta^{(i)}, \; i=1, 2\big\}.$$
Now let $f$ be a polynomial on $\T^2$. By the Fubini theorem and \eqref{DLP-T} applied to the first variable $z_1\in\T$, we get
\begin{align*}
 \|f\|_{L_p(\T^2)}^p
 =\int_{\T}dz_2\int_{\T}|f(z_1, z_2)|^pdz_1
 \le \big(\mathsf{L}^{\wt\Delta}_{t,p, 1}\big)^p\int_{\T}dz_2\int_{\T}\big(\sum_{R^{(1)}\in\wt\Delta^{(1)}}|S_{R^{(1)}}\big(f(\cdot, z_2)\big)(z_1)|^2\big)^{\frac{p}2}dz_1.
 \end{align*}
Let $K=\el_2(\wt\Delta^{(1)})$ equipped with the canonical basis $\{e_{R^{(1)}}\}_{R^{(1)}\in\wt\Delta^{(1)}}$. Then
  $$\big(\sum_{R^{(1)}\in\wt\Delta^{(1)}}|S_{R^{(1)}}\big(f(\cdot, z_2)\big)(z_1)|^2\big)^{\frac{1}2}=\big\|\sum_{R^{(1)}\in\wt\Delta^{(1)}}S_{R^{(1)}}\big(f(\cdot, z_2)\big)(z_1)e_{R^{(1)}}\big\|_K.$$
For each fixed $z_1$, we apply Lemma~\ref{MZ} to the $K$-valued function on the right hand side in the variable $z_2$ in order to infer
 \begin{align*}
 \int_{\T}&\big\|\sum_{R^{(1)}\in\wt\Delta^{(1)}}S_{R^{(1)}}\big(f(\cdot, z_2)\big)(z_1)e_{R^{(1)}}\big\|_K^p\,dz_2\\
  &\les \big(\mathsf{L}^{\wt\Delta}_{t,p, 1}\big)^p\int_{\T}\Big(\sum_{R^{(2)}\in\wt\Delta^{(2)}}\big\|S_{R^{(2)}}\big[\sum_{R^{(1)}\in\wt\Delta^{(1)}}S_{R^{(1)}}\big(f(\cdot, z_2)\big)(z_1)e_{R^{(1)}}\big]\big\|_K^2\Big)^{\frac{p}2}\,dz_2\\
 &=\big(\mathsf{L}^{\wt\Delta}_{t,p, 1}\big)^p\int_{\T}S^{\wt\Delta}(f)(z_1, z_2)^p\,dz_2.
   \end{align*}
 Combining the previous inequalities, we get
   $$\|f\|_{L_p(\T^2)}\les \big(\mathsf{L}^{\wt\Delta}_{t,p, 1}\big)^2   \|S^{\wt\Delta}(f)\|_{L_p(\T^2)},$$
  whence $\mathsf{L}^{\wt\Delta}_{t,p, 2}\les \big(\mathsf{L}^{\wt\Delta}_{t,p, 1}\big)^2$.
\end{proof}

\begin{rk}
 Pichorides  \cite{Pichorides92} studied the first inequality of \eqref{DLP-T} for $d=1$ restricted to functions in the Hardy space, and proved that the corresponding constant is of $p'$ as $p\to1$. Combined with Remark~\ref{MZbis}, the above proof shows that Pichorides' result extends to higher dimensions, we thus recover a result of \cite{BakasRS} (see also \cite{Bakas21} for related results).
\end{rk}

We conclude the paper with the proof of Corollary~\ref{DLP-const}.

\begin{proof}[Proof of Corollary~\ref{DLP-const}]
 By Theorem~\ref{RvsT}, Theorem~\ref{multi-parameter}  and the known results mentioned in the historical comments at the end of section~\ref{Introduction}, we need only to show $p^{\frac{d}2}\les\mathsf{L}^{\Delta}_{t,p, 1}\les  p$ for $2\le p<\8$. The first inequality is proved by using lacunary series as in the last part of the proof of Theorem~\ref{PH ML}. It remains to show the second, that is, we must prove
  $$\|f\|_p\les p\|S^{\Delta}(f)\|_p,\quad f\in L_p(\real).$$
 To this end, it suffices to consider a (nice) function $f$ whose Fourier transform is supported in $\real_+$. Fixing such an $f$, let $S_k(f)=S_R(f)$ for $R=[2^{k-1},\, 2^k)$, i.e.,  $\wh{S_k(f)}=\un_{[2^{k-1},\, 2^k)}\wh f$.

 We will use the smooth version of $S^{\Delta}$. Let $\f$ be a $C^\8$ function on $\real$ whose Fourier transform is supported in $\{\xi: \frac12<|\xi|<4\}$ and satisfies
 $$\sum_{k\ent}\wh\f(2^{-k}\xi)=1,\quad \xi\in\real\setminus\{0\}.$$
Then the smooth version of $S^{\Delta}$ is the discretization of the $g$-function in \eqref{G-function}:
 $$
 G^{\f}_{\rm dis}(g)(x)=\Big(\sum_{k\in\ent}|\f_k*g(x)|^2 \Big)^{\frac12}, \quad x\in\real
 $$
 for any (nice) function $g$ on $\real$.
Then
 \begin{align*}
 \int_{\real} f(x)\,\overline{g(x)}dx
 &= \sum_{k\in\ent}\sum_{j=k-1}^{k+1}\int_{\real} S_k(f)(x)\,\overline{\f_j*g(x)}dx.
 \end{align*}
Thus by the H\"older inequality,
 $$
 \Big|\int_{\real} f(x)\overline{g(x)}dx\Big|
 \le 3\|S^{\Delta}(f)\|_p \|G^{\f}_{\rm dis}(g)\|_{p'}.$$
However, it is well known that
 $$\|G^{\f}_{\rm dis}(g)\|_{p'}\les p \|g\|_{p'}.$$
This is also the discrete analogue of Theorem~\ref{fML} (i). It then follows that
  $$
 \Big|\int_{\real} f(x)\overline{g(x)}dx\Big|\les p\|S^{\Delta}(f)\|_p \|g\|_{p'}.$$
Taking the supremum over all $g$ with $\|g\|_{p'}\le1$ yields the desired inequality on $f$.
\end{proof}


\bigskip \n{\bf Acknowledgements.}  I am indebted to Assaf Naor  for many inspiring communications that are special impulse to my research carried out here as well as largely motivated  my work \cite{HFC-LPS} (Assaf also asked himself the problem on the optimal orders of the constants in \eqref{CLPt}  in his own research). I am grateful to Odysseas Bakas, Guixiang Hong, Tao Mei and Lixin Yan for useful discussions, and also to Odysseas Bakas and Hao Zhang for allowing me to include their estimate $\mathsf{L}^{\Delta}_{t,p, d}\les_d p^d$ in Corollary~\ref{DLP-const} as well as its proof. Finally, I wish to thank the anonymous referee for valuable suggestions. This work is partially supported by the French ANR project (No. ANR-19-CE40-0002).
\bigskip



\begin{thebibliography}{10}


\bibitem{Bakas19}
O.~Bakas. Endpoint mapping properties of the {L}ittlewood-{P}aley square function. \textit{Colloq. Math.} 157 (2019), 1--15.

\bibitem{Bakas21}
O.~Bakas. On a problem of Pichorides. \textit{J. Geom. Anal.} 31 (2021), 12637--12639.

\bibitem{BakasRS}
O.~Bakas, S.~Rodr\'{\i}guez-L\'opez, and A. A.~Sola. Multi-parameter extensions of a theorem of Pichorides. \textit{Proc. Amer. Math. Soc.} 147 (2019),1081--1095

\bibitem{Bourgain83}
J.~Bourgain. Some remarks on Banach spaces in which martingale differences are unconditional. \textit{Ark. Mat.} 21 (1983), 163--168.

\bibitem{Bourgain85}
J.~Bourgain. On square functions on the trigonometric system. \textit{Bull. Soc. Math. Belg. S\'er. B} 37 (1985), 20--26.

\bibitem{Bourgain89}
J.~Bourgain. On the behavior of the constant in the {L}ittlewood-{P}aley inequality. \textit{Lecture Notes in Math.} 1376 (1989), 202--208.



\bibitem{CWW}
S-Y. A.~Chang, J.~M.~Wilson, and T. Wolff. Some weighted norm inequalities concerning the Schr\"odinger operator. \textit{Comm. Math. Helv.} 60 (1985), 217-246.

\bibitem{CHLM}
C. Y.~Chuah, Y.~Han, Z.~Liu, and T. Mei. Paley's inequality for nonabelian groups. Preprint 2021 (arXiv).

\bibitem{CDWL}
P.~Chen, X. T.~Duong, L.~Wu, and L.~Yan. Exponential-square integrability, weighted inequalities for the square functions associated to operators, and
applications. \textit{Int. Math. Res. Not.}  23 (2021), 18057-18117.

\bibitem{CXY2012}
Z.~Chen, Q.~Xu and Z.~Yin. Harmonic Analysis on Quantum Tori. {\em  Commun. Math. Phys.}  322 (2013), 755--805.

\bibitem{Leeuw}
K.~ de Leeuw. On $L^p$ multipliers.  \textit{Ann.  Math.} 91 (1965),  364--379.

\bibitem{FS}
C.~Fefferman, and E.M. Stein. $H^p$ spaces of several variables. \textit{Acta Math.} 129 (1972), 137-193.

\bibitem{FP}
R.~Fefferman, and J.~Pipher. Multiparameter operators and sharp weighted inequalities.
 \textit{Amer. J. Math.} 119 (1997), 337--69.

\bibitem{GRF}
J.~Garc\'{\i}a-Cuerva, and J.L. Rubio de Francia. \textit{Weighted norm inequalities and related topics}. North-Holland Publishing Co., Amsterdam, 1985.

\bibitem{GY}
R.~Gong, and L.~Yan. Weighted $L_p$ estimates for the area integral associated to selfadjoint operators. \textit{Manuscripta Math.} 144 (2014), 25--49.

\bibitem{Grafakos}
L.~Grafakos.  \textit{Classical Fourier analysis}. Second edition.  Springer, New York, 2008.


\bibitem{journe}
J-L.~Journ\'e. Calder\'on-Zygmund operators on product spaces.  \textit{Rev. Mat. Iberoam.} 1 (1985), 55--91.

\bibitem{Lerner19}
A.~Lerner. Quantitative weighted estimates for the Littlewood-Paley square function and Marcinkiewicz multipliers.  \textit{Math. Res. Lett.}  26 (2019), 537--556.

\bibitem{Mei2007}
 T.~Mei. Operator valued Hardy spaces. {\it Memoirs Amer. Math. Soc.} 881 (2007), vi+64 pp.

\bibitem{Meyer}
P-A.~Meyer. D\'emonstration probabiliste de certaines in\'egalit\'es de Littlewood-Paley I: les in\'egalit\'es classiques . \textit{S\'em. Probab. de Strasbourg} 10 (1976), 125-141.


\bibitem{Osipov11}
N.~N.~Osipov. A one-sided {L}ittlewood-{P}aley inequality in {$\Bbb R^n$} for {$0<p\le2$}. \textit{J. Math. Sci.} 172 (2011), 229--242.

\bibitem{Pichorides90}
S.~Pichorides. On the {L}ittlewood-{P}aley square inequality. \textit{Colloq. Math.} 60/61 (1990), 687--691.

\bibitem{Pichorides92}
S.~Pichorides. A remark on the constants of  the {L}ittlewood-{P}aley square inequality. \textit{Proc. Amer. Math. Soc.} 114 (1992), 787--789.

\bibitem{Pisier}
G.~Pisier. Les in\'egali\'es de Khintchine-Kahane, d'apr\`es C. Borell.  \textit{S\'eminaire d'Analyse Fonctionnelle, Ecole Polytechnique}, 1977-1978, Exp. 7.

\bibitem{rubio}
J. L.~Rubio de Francia. A Littlewood-Paley inequality for arbitrary intervals.  \textit{Rev. Mat. Iberoam.} 1 (1985), 1--14.

\bibitem{Rudin}
W.~Rudin. Trigonometric series with gaps.   \textit{J. Math. Mech.} 2 (1960), 203--227.


\bibitem{stein-square}
E.M.~Stein.  The development of square functions in the work of A. Zygmund. \textit{Bull. Amer. Math. Soc.} 7 (1982), 359-376.

\bibitem{Stein1993}
E. M. Stein. {\em Harmonic analysis}.  Princeton University Press, Princeton, 1993.

\bibitem{Stein-Weiss}
E. M.~Stein, and G.~Weiss. {\it Introduction to Fourier Analysis on Euclidean Spaces}. Princeton University Press, Princeton, 1975.

\bibitem{Uch}
A.~Uchiyama. A constructive proof of the Fefferman-Stein decomposition of BMO $({\bf R}^{n})$. \textit{Acta Math.} 148 1(1982), 215--241.



\bibitem{Wilson07}
J.~M.~Wilson. The intrinsic square function. \textit{Rev. Mat. Iberoam.} 23 (2007), 771--791.

\bibitem{Wilson08}
J.~M.~Wilson. \textit{Weighted {L}ittlewood-{P}aley theory and exponential-square  integrability}. Lecture Notes in Math. Springer, Berlin,  1924 (2008), pp. xiv+224.

\bibitem{XXX}
 R.~Xia, X.~Xiong, and Q.~Xu. Characterizations of operator-valued Hardy spaces and applications to harmonic analysis on quantum tori. \textit{Adv. Math.}  291 (2016), 183--227.

\bibitem{LP0}
 Q.~Xu. Littlewood-{P}aley theory for functions with values  in uniformly convex spaces. \textit{J. Reine Angew. Math.} 504 (1998), 195-226.

\bibitem{HFC-LPS}
 Q.~Xu. Holomorphic functional calculus and vector-valued Littlewood-Paley-Stein theory for semigroups. Preprint 2021 (arXiv).

\bibitem{XZ}
Z.~Xu, and H. Zhang. From the Littlewood-Paley-Stein inequality to the Burkholder-Gundy inequality. Preprint 2021 (arXiv).



\end{thebibliography}
\end{document}